\documentclass[a4paper,11pt]{article}
\usepackage{amsmath}
\usepackage{amsfonts}
\usepackage{amssymb}
\usepackage{amsthm}
\usepackage{verbatim}
\usepackage{pstricks}
\usepackage[labelsep=none]{caption}

\newtheorem{theorem}{Theorem}[section]
\newtheorem{corollary}[theorem]{Corollary}
\newtheorem{lemma}[theorem]{Lemma}
\newtheorem{proposition}[theorem]{Proposition}

\theoremstyle{definition}
\newtheorem{definition}[theorem]{Definition}
\newtheorem{remark}[theorem]{Remark}

\newcommand{\crcle}{\mathbb{S}^1}  						
\newcommand{\sphere}{\mathbb{S}^2}  						
\newcommand{\Sphere}{\mathbb{S}^3}  						
\newcommand{\bbb}[1]{\mathbb{B}_{#1}}						
\newcommand{\intvl}{\mathrm{I}}              		
\DeclareMathOperator*{\nhd}{\mathcal{N}} 				
\DeclareMathOperator*{\Int}{int}  							

\DeclareMathOperator{\ms}{MS}    								
\newcommand{\kc}[1]{\mathcal{K}(#1)}    				
\DeclareMathOperator*{\esd}{Esd} 								
\DeclareMathOperator*{\dist}{d}   							

\DeclareMathOperator{\V}{V}   									
\DeclareMathOperator{\E}{E}   									
\newcommand{\graph}[2]{G^{\mathcal{#1}}(#2)}		

\begin{document}
\title{Minimal genus Seifert surfaces for alternating links}
\author{Jessica E. Banks}
\date{}
\maketitle

\begin{abstract}
We give a complete proof of results announced by Hirasawa and Sakuma describing explicitly the Kakimizu complex of a non-split, prime, special alternating link.
\end{abstract}

\tableofcontents


\section{Introduction}

The Kakimizu complex of a link records the structure of the set of minimal genus Seifert surfaces for the link. For a non-split, alternating link, the definition is analogous to that of the complex of curves of a compact, orientable surface.

\begin{definition}[\cite{MR1177053} p225]
For a non-split, alternating link $L$, the \textit{Kakimizu complex} $\ms(L)$ of $L$ is a simplicial complex, the vertices of which are the ambient isotopy classes of minimal genus Seifert surfaces for $L$. Distinct vertices $R_0,\ldots,R_n$ span an $n$--simplex exactly when they can be realised disjointly.
\end{definition}

The Kakimizu complex can be defined for links in general, but not using this form of the definition.
It has been calculated explicitly in certain simple cases (see, for example, \cite{MR2131376}). However, doing so in general is complicated by the difficulty of controlling the multiple surfaces that form the vertices of a simplex. 
One way in which we might hope to do so is by expressing these surfaces in terms of a diagram for $L$. With this in mind, alternating links are a natural class to study, as applying Seifert's algorithm to any alternating diagram gives a minimal genus Seifert surface (\cite{MR860665} Theorem 4). In particular, the subclass of special alternating links is very well behaved.

\begin{definition}
A \textit{special} link diagram $D$ is one in which every Seifert circle is innermost in $\sphere$. This determines a division of the regions of $D$ into \textit{black regions} (those which make up the Seifert surface given by Seifert's algorithm) and \textit{white regions}. Thus the surface given by applying Seifert's algorithm to $D$ is always uniquely defined.
\end{definition}

In \cite{MR1664976} Hirasawa and Sakuma give a complete description of $\ms(L)$ for a prime link $L$ with a reduced, special alternating diagram $D$ via the following result. 
The definition of the complex $\kc{D}$ is given in Section \ref{hssection}.

\begin{theorem}[\cite{MR1664976} Theorems 1.5, 1.1]\label{isomorphismtheorem}
There is a natural isomorphism between $\ms(L)$ and $\kc{D}$. In particular, every minimal genus Seifert surface for $L$ is given by applying Seifert's algorithm to some special alternating diagram for $L$.
\end{theorem}

From this they deduce contractibility of $\ms(L)$ for such links $L$, through the following.

\begin{theorem}[\cite{MR1664976} Theorem 1.6]
If $L$ is a prime link with a special alternating diagram $D$ then $\ms(L)$ is homeomorphic to a disc, the dimension of which can be calculated from $D$.
\end{theorem}

\noindent Only `the idea of the proof' of these theorems is included in \cite{MR1664976}. In \cite{MR2531146} it says the following.
\begin{quotation}
\textit{[Sakuma]} proves the contractibility of $\ms(K)$ when $K$ is a special arborescent link. In his joint paper with Hirasawa \textit{[\cite{MR1664976}]}, contractibility when $K$ is a prime, special, alternating link is announced. Together, this partially verifies a challenging conjecture of Kakimizu's, asserting among other things that $\ms(K)$ is always contractible.
\end{quotation}

The question of the contractibility of the Kakimizu complex for a general link has been answered by the following theorem of Przytycki and Schultens.

\begin{theorem}[\cite{2010arXiv1004.4168P} Theorem 1.1]\label{contractiblethm}
The Kakimizu complex of a link is contractible.
\end{theorem}

\begin{definition}
A metric is defined on the vertices of $\ms(L)$. The distance between two vertices is the distance in the 1--skeleton of $\ms(L)$ when every edge has length 1.
\end{definition}

Here we provide a full proof of the results of \cite{MR1664976}, with a view to determining $\ms(L)$ for a general alternating link $L$. Theorem \ref{contractiblethm} is needed in this proof.

This paper is organised as follows. In Section \ref{hssection} we recall the necessary definitions from \cite{MR1664976}. In Section \ref{specialsection} we define a method for describing minimal genus Seifert surfaces relative to an alternating diagram. Section \ref{inductionsection} addresses Theorem 1.1 of \cite{MR1664976}, and Section \ref{disjointnesssection} extends this to a proof of \cite{MR1664976} Theorem 1.5. Finally, Theorems 1.2, 1.6 of \cite{MR1664976} are proved in Sections \ref{nonspecialsection}, \ref{productssection} respectively.

I wish to thank Marc Lackenby for helpful conversations and encouragement.
I am also grateful to Richard Webb and Jacob Rasmussen, through whom I learned of \cite{MR1664976}.

\section{The paper of Hirasawa--Sakuma}\label{hssection}

There are a number of definitions we will need from \cite{MR1664976}. We will change some of the notation and terminology.
Throughout this paper we will only consider non-split, oriented links.

\bigskip

Let $L$ be a prime link with a reduced, special alternating diagram $D$. Consider the planar graph $\mathcal{G}$ that has a vertex in each black region and an edge through each crossing of $D$.
Note that $\mathcal{G}$ contains no loops. 
It may be that distinct edges $e,e'\in\E(\mathcal{G})$ bound a bigon region of $\sphere\setminus\mathcal{G}$. If this occurs, remove one of $e,e'$. Repeat this until $\sphere\setminus\mathcal{G}$ has no bigon regions.

Suppose there exists a simple closed curve $\rho$ in $\sphere$ such that $\rho$ consists of an edge $e$ of $\mathcal{G}$ together with an arc $\rho'$ that only meets $\mathcal{G}$ at its endpoints and such that $\sphere\setminus(\mathcal{G}\cup \rho')$ has no bigon regions. Add a new edge to $\mathcal{G}$ along $\rho'$. Repeat this as many times as possible. Since any two such arcs $\rho'$ can be made disjoint on their interiors, the result of this process is well-defined given $D$.

\begin{definition}
Define $\graph{F}{D}$ to be the graph that results from this process. 
\end{definition}

Figure \ref{hsdefnspic1} shows a digram $D_{\alpha}$ of a link $L_{\alpha}$ together with the graph $\graph{F}{D_{\alpha}}$.
\begin{figure}[htbp]
\centering
\input{pictexfiles/hsdefnspic1}
\caption{\label{hsdefnspic1}}
\end{figure}

\begin{definition}
A \textit{flype circle} $\phi$ is a simple closed curve in $\sphere$ that meets the link diagram $D$ as shown in Figure \ref{hsdefnspic2}a, where the tangles $A$ and $B$ each contain at least one crossing. The flype circle $\phi$ determines a flype that changes $D$ to another special alternating diagram of $L$, as shown in Figure \ref{hsdefnspic2}b.  This change is realised by an isotopy of $\Sphere$.
\begin{figure}[htbp]
\centering
(a)
\input{pictexfiles/hsdefnspic2a}
(b)
\input{pictexfiles/hsdefnspic2b}
\caption{\label{hsdefnspic2}}
\end{figure}
\end{definition}

Let $\phi$ be a flype circle in $D$. 
Let $D_{\phi}$ be the diagram given by performing the flype defined by $\phi$ on $D$. Let $R,R_{\phi}$ be the surfaces given by applying Seifert's algorithm to $D,D_{\phi}$ respectively. Finally, let $R'$ be the preimage of $R_{\phi}$ under the flype.

\begin{definition}
The flype circle $\phi$ is called \textit{inessential} if $R'$ is ambient isotopic to $R$, and \textit{essential} otherwise.
\end{definition}

If $\phi$ lies in the white regions of the diagram where it meets the crossing of $D$ then $\phi$ is inessential. From now on we will generally ignore all such flype circles. That is, we assume that $\phi$ lies on the surface $R$ where it meets the crossing.

\begin{definition}
We call the crossing through which $\phi$ passes the \textit{flype crossing}, and the arc of $\phi$ disjoint from $R$ the \textit{flype arc}. 
\end{definition}

White bigons in a special alternating diagram signify plumbed on Hopf bands in the Seifert surface given by Seifert's algorithm (see Lemma \ref{hopfbandlemma}). Because the Hopf link is fibred, a flype that interchanges the two crossings of a white bigon is inessential. Thus $\phi$ is inessential if the flype crossing and flype arc of $\phi$ differ only by a line of white bigons. We will see later that the converse is also true.

If $\phi$ is essential, the flype crossing and flype arc correspond to distinct edges $e,e'$ of $\graph{F}{D}$ with the same endpoints. We will also refer to the simple closed curve $e\cup e'$ as the flype circle $\phi$, provided we retain knowledge of which edge is the flype crossing and which the flype arc. Note that distinct flype circles $\phi$ and $\phi'$ in $D$ may give the same flype circle in $\graph{F}{D}$. In this case, the flype crossings of $\phi$ and $\phi'$ differ by at most a line of white bigons, as do the flype arcs. Hence the diagrams $D_{\phi}$ and $D_{\phi'}$ are the same.

\begin{definition}
Call such flype circles \textit{equivalent}.
\end{definition}

\begin{definition}
Let $v,v'\in\V(\graph{F}{D})$ such that $v,v'$ are joined by at least two edges of $\graph{F}{D}$. Then the subgraph of $\graph{F}{D}$ consisting of $v,v'$ and all edges joining them
is called a \textit{$\theta$--graph}. For an edge $e$ of a $\theta$--graph, denote by $e^{\theta}$ the $\theta$--graph containing $e$.
Define $\theta(D)$ to be the subgraph of $\graph{F}{D}$ that is the union of all $\theta$--graphs in $\graph{F}{D}$. 
\end{definition}

\begin{definition}
Each vertex of $\graph{F}{D}$ inherits an orientation (clockwise or anticlockwise) from the Seifert circle it lies inside. This orientation extends to a transverse orientation of each edge of $\graph{F}{D}$. We define the \textit{positive} and \textit{negative sides} of each edge, such that the normal points from the negative side to the positive side. See Figure \ref{hsdefnspic3}. 
We similarly define the \textit{positive} and \textit{negative sides} of each crossing in $D$.
\begin{figure}[htbp]
\centering
\input{pictexfiles/hsdefnspic3}
\caption{\label{hsdefnspic3}}
\end{figure}
\end{definition}

\begin{definition}
Given an essential flype circle $\phi$ in $D$, define the \textit{positive side of $\phi$} to be the component of $\sphere\setminus\phi$ that meets the positive side of the flype crossing of $\phi$. The other component is called the \textit{negative side}.
\end{definition}

Each edge $e\in\E(\theta(D))$ inherits a weight $w(e)\in\mathbb{N}\cup\{0\}$, given by the number of crossings in $D$ that correspond to $e$. Define $w^{\theta}(e^{\theta})=\sum_{e'\in e^{\theta}}{w(e')}$.
Number the edges of $\theta(D)$ as $e_1,\ldots,e_n$.
These weights are used to define a simplicial complex $\kc{D}$. First, if $n\geq 1$, set 
\[
\begin{split}
\V(\kc{D})=\Bigg\{(w_1,\ldots,w_n)\in \mathbb{Z}^n {}:{} &  w_i\geq 0 \textrm{ and }\\ & \sum_{\left\{j:e_j\in e_i^{\theta}\right\}}{w_j}=w^{\theta}(e_i^{\theta}) \textrm{ for } 1\leq i\leq n \Bigg\}.
\end{split}
\]
If $D$ contains no $\theta$--graphs, and so $\theta(D)$ is empty, take $\kc{D}$ to consist of a single vertex.

\begin{definition}
Let $\tilde{\theta}(D)$ be the planar graph obtained from $\theta(D)$ by cutting apart vertices to make the $\theta$--graphs disjoint.
\end{definition}

Note that a region of $\sphere\setminus\tilde{\theta}(D)$ corresponds to a union of regions of $\sphere\setminus\theta(D)$. 
We will refer to these as \textit{a region of $\theta(D)$}.

\begin{definition}
Let $r$ be a region of $\theta(D)$. Define the \textit{positive boundary} $\partial^+ r$ of $r$ to be the edges of $\theta(D)$ which $r$ meets exactly on the negative side, and the \textit{negative boundary} $\partial^- r$ to be those it meets exactly on the positive side.

Using this, the region $r$ defines a map $r^{\theta}$ from a subset of $\V(\kc{D})$ to $\V(\kc{D})$, given by  $r^{\theta}(w_1,\ldots,w_n)=(w'_1,\ldots,w'_n)$, where
\[
w'_i=
\begin{cases}
w_i+1 & \textrm{if } e_i\in \partial^+ r\\
w_i-1 & \textrm{if } e_i\in \partial^- r\\
w_i & \textrm{else.}
\end{cases}
\]
Thus $r^{\theta}(w_1,\ldots,w_n)$ is defined when $w_i>0$ for all $e_i\in\partial^- r$.

We describe the process of applying $r^{\theta}$ to $v\in\V(\kc{D})$ as \textit{adding the region $r$ to $v$}.
\end{definition}

The higher-dimensional simplices of $\kc{D}$ are defined as follows. A set of distinct vertices $v_0,\ldots,v_k$ spans a $k$--simplex if there is a choice of labelling of the regions of $\theta(D)$ as $r_1,\ldots,r_m$ such that $r_j^{\theta}(\cdots r_1^{\theta}(v_0)\cdots)$ is defined for $1\leq j\leq m$ and each $v_i$ occurs as $r_j^{\theta}(\cdots r_1^{\theta}(v_0)\cdots)$ for some $j$. 
That is, every simplex can be extended to an $(m-1)$--simplex, and we can cycle through the vertices of an $(m-1)$--simplex by adding each region once, in some order.
Note that $r_m^{\theta}(\cdots r_1^{\theta}(v_0)\cdots)=v_0$.

A metric is defined on $\V(\kc{D})$ in the same way as that defined on $\ms(L)$. That is, the distance between two
vertices is the distance in the 1--skeleton of $\kc{D}$ when every edge has length 1.

Note that $\kc{D}$ depends only on $\theta(D)$. A vertex $v$ of $\kc{D}$ specifies a reduced, special alternating diagram $D_v$ of $L$, to which Seifert's algorithm can be applied. $D_v$ is obtained from $D$ by a canonical set of flypes. A flype on $D$ can result in either of two diagrams, which differ by turning over the whole diagram. 
However, the surface given by Seifert's algorithm is independent of this choice. This gives a map $\mathcal{A}$ from $\V(\kc{D})$ to $\V(\ms(L))$.
\begin{figure}[htbp]
\centering
\input{pictexfiles/hsdefnspic4}
\caption{\label{hsdefnspic4}}
\end{figure}

Figure \ref{hsdefnspic4} shows the graph $\theta(D_{\alpha})$ for the diagram $D_{\alpha}$ in Figure \ref{hsdefnspic1}, marked with the weights that come from $D_{\alpha}$. The complex $\kc{D_{\alpha}}$ is 3--dimensional, and this vertex lies in four 3--simplices, corresponding to the cycles
\[
\begin{split}
(1,0,2,0,1)\rightarrow (1,0,3,0,0) \rightarrow (0,1,3,0,0) \rightarrow (1,0,2,1,0) \rightarrow (1,0,2,0,1),\\
(1,0,2,0,1)\rightarrow (0,1,2,0,1) \rightarrow (0,1,3,0,0) \rightarrow (1,0,2,1,0) \rightarrow (1,0,2,0,1),\\
(1,0,2,0,1)\rightarrow (0,1,2,0,1) \rightarrow (1,0,1,1,1) \rightarrow (1,0,2,1,0) \rightarrow (1,0,2,0,1),\\
(1,0,2,0,1)\rightarrow (0,1,2,0,1) \rightarrow (1,0,1,1,1) \rightarrow (1,0,1,0,2) \rightarrow (1,0,2,0,1).
\end{split}
\]

\section{Special form}\label{specialsection}

\subsection{Definition}

Let $L$ be a special alternating link with a special alternating diagram $D\subset\sphere\subset\Sphere$.

\begin{definition}
Let $\bbb{a}$ be the 3--ball lying above $\sphere$ in $\Sphere$, $\bbb{b}$ the 3--ball lying below $\sphere$, and $P$ the set of midpoints of edges of $D$. We call a section of $L$ between two consecutive points of $P$ an \textit{arc of $L$}. 
\end{definition}

By an isotopy, $L$ can be arranged such that $L\cap\sphere=P$, with overcrossing arcs of $L$ lying in $\bbb{a}$ and undercrossing arcs lying in $\bbb{b}$ (as in \cite{MR860665}).
Let $R$ be the Seifert surface given by applying Seifert's algorithm to $D$. 
Let $\nhd(L)$ be a regular neighbourhood of $L$ in $\Sphere$ and $\nhd_{L}(R)$ a product neighbourhood of $R$ in $\Sphere\setminus\nhd(L)$. Then $\nhd(R)=\nhd(L)\cup\nhd_{\!L}(R)$ is a regular neighbourhood of $R$ in $\Sphere$.

Consider an incompressible Seifert surface $R'$ for $L$ that is disjoint from $R$. By an isotopy, $\partial R'$ can be made to run along the opposite side of $\partial\nhd(L)$ to $\partial R$ (again as in \cite{MR860665}). In a neighbourhood of a crossing, $\partial R$ and $\partial R'$ are as shown in Figure \ref{specialpic1}.
\begin{figure}[htbp]
\centering
\input{pictexfiles/specialpic1}
\caption{\label{specialpic1}}
\end{figure}

Now isotope $R'$, keeping $\partial R'$ fixed and keeping $R'$ disjoint from $R$, to minimise $|R'\cap \sphere|$. Note that $R'\cap\sphere$ is disjoint from the black regions of $D$. Since $R'$ is incompressible and $\Sphere\setminus\nhd(R)$ is irreducible, $R'\cap \sphere$ contains no closed components, so it consists of arcs with their endpoints on $\partial R'\cap\sphere$. Similarly, no such arc has both endpoints at the same point of $R'\cap \sphere$. On the other hand, every point of $R'\cap \sphere$ is an endpoint of at least one arc. We will identify these points with $P$. 

Suppose there exists $p\in P$ such that at least two arcs of $R'\cap\sphere$ have an endpoint at $p$. Then $R'$ can be isotoped to reduce the number of endpoints of arcs at $p$ by 2. This reduces the number of arcs of $R'\cap\sphere$ and any closed curve created can be removed, which contradicts the minimality of $|R'\cap\sphere|$. Hence each point $p\in P$ is the endpoint of exactly one arc.

\begin{definition}
Call a set of disjoint arcs lying in the white regions of $D$ with exactly one arc endpoint at each point of $P$ a \textit{set of $P$--arcs}.
\end{definition}

Consider the (probably disconnected) surface $R'_a=R'\cap \bbb{a}$. Its boundary $\partial R'_a$ projects to disjoint simple closed curves in $\sphere$, together composed exactly of the set of $P$--arcs and the overcrossing arcs of $L$. Unless $L$ is the unknot, each component includes both overcrossings and $P$--arcs, alternating around the curve. 

We wish to show that every component of $R'_a$ is a disc. To see this, first isotope each component downwards so that its boundary lies in $\sphere$.
Now choose a closed curve $C$ in $\partial R'_a$ that is innermost in $\sphere$.  It bounds a disc in $\bbb{a}$ that is otherwise disjoint from $R'_a$, and so bounds a disc in $R'$. This disc cannot lie below $C$ in $R'$ since part of $C$ lies on $\partial R'$, so it lies above $C$.
If the interior of the disc meets $\sphere$, then it also meets $\partial R'$, which cannot be the case.
Thus the disc lies in $\bbb{a}$. By isotoping it down towards $\sphere$, we see that for our current argument we may discard this component of $R'_a$. Thus inducting on $|\partial R'_a|$ gives the required result.

A similar argument holds for $R'_b=R'\cap \bbb{b}=R'\setminus\Int_{R'}(R'_a)$. Hence $R'$ is completely specified by the set of $P$--arcs $R'\cap\sphere$ added to the diagram $D$.

\begin{definition}
We will say that a surface defined in this way by a set of $P$--arcs is in \textit{special form}.
\end{definition}

Suppose $D$ has $n$ crossings.
Then there are $n$ $P$--arcs. Let $n_a$ be the number of discs in $R'_a$, and $n_b$ the number in $R'_b$. 
Then $\chi(R')=2n-(2n+n)+(n_a+n_b)=-n+n_a+n_b$.

Note that $n_a=|R'\cap \bbb{a}|=|\partial(R'\cap \bbb{a})|$. Hence $R'$ has minimal genus if and only if $|\partial(R'\cap \bbb{a})|+|\partial(R'\cap \bbb{b})|$ (that is, the number of simple closed curves formed from the $P$--arcs and the arcs of $L$) is maximised.


\subsection{Some surfaces in special form}

Fix a special alternating diagram $D$ of a link $L$. Let $R$ be the Seifert surface given by applying Seifert's algorithm to $D$. We can define a Seifert surface $R'$ in special form by putting a $P$--arc across the negative side of each crossing, as show in Figure \ref{coherencepic1}.
\begin{figure}[htb]
\centering
\input{pictexfiles/coherencepic1}
\caption{\label{coherencepic1}}
\end{figure}
By pushing each $P$--arc close to the corresponding crossing, it is easy to see that $R'=R$. Placing a $P$--arc on the positive side of each crossing again gives a surface equivalent to $R$. It is given by pushing $R'$ through the parallel surface $R$ to the other side.

Figure \ref{coherencepic2} shows two digrams $D_{\beta}$ and $D_{\beta}^*$ of a link $L_{\beta}$ that differ by a flype along the flype circle $\phi$ shown.
\begin{figure}[htb]
\centering
\hspace{1.3em}
\input{pictexfiles/coherencepic2a}
\hspace{1.3em}
\input{pictexfiles/coherencepic2b}
\caption{\label{coherencepic2}}
\end{figure}
Let $R_{\beta}, R_{\beta}^{*}$ be the surfaces given by applying Seifert's algorithm to $D_{\beta}, D_{\beta}^*$ respectively. Consider the effect on $R_{\beta}^{*}$ of an isotopy that takes $D_{\beta}^*$ to $D_{\beta}$. By inspection, $R_{\beta}^{*}$ can be put into special form with respect to $D_{\beta}$ as shown in Figure \ref{coherencepic3}a. Figure \ref{coherencepic3}b shows $\partial(R_{\beta}^{*}\cap \bbb{a})$. 
\begin{figure}[htbp]
\centering
(a)
\input{pictexfiles/coherencepic3a}
(b)
\input{pictexfiles/coherencepic3b}
\caption{\label{coherencepic3}}
\end{figure}

The special form of $R_{\beta}^{*}$ has the following description. The flype arc of $\phi$ is a $P$--arc. Every crossing $c\neq c_{\phi}$ (where $c_{\phi}$ is the flype crossing of $\phi$) has a $P$--arc across it. This $P$--arc crosses $c$ on its negative side if $c$ is on the positive side of $\phi$, and on its positive side if $c$ lies on the negative side of $\phi$.

We can extend this description to the case of more than one flype circle on a link diagram $D$, provided we can consistently define the notion of being on the positive/negative side of the flype circles.

\begin{definition}
Call a set of distinct flype circles with this property \textit{coherent}. Otherwise call them \textit{incoherent}.
\end{definition}

\begin{remark}
Coherent flype circles must be pairwise disjoint. Figure \ref{coherencepic4} shows a set of incoherent disjoint flype circles on the diagram $D_{\alpha}$.
\begin{figure}[htb]
\centering
\input{pictexfiles/coherencepic4}
\caption{\label{coherencepic4}}
\end{figure}
\end{remark}

Using the above description, a set of coherent flype circles gives a set of arcs in the white regions of $D$. We need to check whether they form a set of $P$--arcs.

\begin{lemma}
The arcs defined by a coherent set of flype circles form a set of $P$--arcs.
\end{lemma}
\begin{proof}
This will be true if and only if exactly one endpoint of an arc lies on any given edge $\varepsilon$ of $D$. By counting the arcs we see that it is enough to check that at least one arc has an endpoint on $\varepsilon$.

If a flype circle crosses $\varepsilon$, then an arc has an endpoint on $\varepsilon$. Assume otherwise. Let $c_1,c_2$ be the crossings at the ends of $\varepsilon$ such that $\varepsilon$ lies on the positive side of $c_1$ and on the negative side of $c_2$. Then both, one or neither of $c_1,c_2$ is a flype crossing. Since $\varepsilon$ is not crossed by any flype circle, $c_1$ and $c_2$ cannot both be flype crossings, as otherwise the flype circles would be incoherent (see Figure \ref{coherencepic5}).
\begin{figure}[htb]
\centering
\input{pictexfiles/coherencepic5}
\caption{\label{coherencepic5}}
\end{figure}
Suppose neither is a flype crossing. They then lie on the same side of the flype circles, say the negative side. Then the arc across the positive side of $c_1$ has an endpoint on $\varepsilon$. Finally suppose, without loss of generality, that $c_1$ is a flype crossing, while $c_2$ is not. Then $c_2$ lies on the positive side of the flype circles, so has an arc across it on the negative side. This arc has an endpoint on $\varepsilon$.
\end{proof}

\begin{definition}
Given a special alternating link diagram $D$, say that a Seifert surface $R$ for $D$ is \textit{admissible at $D$} if $R$ can be put into special form with the following description.
\begin{itemize}
	\item There is a set of coherent flype circles in $D$.
	\item The flype arc of each flype circle is a $P$--arc.
	\item Every crossing of $D$ that is not a flype crossing has a $P$--arc across it, on the negative side if it lies on the positive side of the flype circles, and on the positive side if it lies on the negative side of the flype circles.
\end{itemize}
Call such a description of $R$ \textit{admissible}.
\end{definition}

\begin{remark}
We allow the case where there are no flype circles. We can then take the crossings of $D$ to all lie on the positive side of the flype circles, or all lie on the negative side. As we have seen, both these special forms describe the surface given by applying Seifert's algorithm to $D$.
\end{remark}

To construct a set of $P$--arcs, it is in fact sufficient to specify a coherent set of flype circles in $\theta(D)$, provided the edge of each such flype circle chosen as a flype crossing has at least one crossing in $D$ corresponding to it. To see this, notice by inspection that equivalent flype circles create the same set of $P$--arcs (see Figure \ref{coherencepic6}).
\begin{figure}[htb]
\centering
\input{pictexfiles/coherencepic6a}
\input{pictexfiles/coherencepic6b}
\caption{\label{coherencepic6}}
\end{figure}

\begin{lemma}\label{flypeslemma1}
Let $R'$ be a surface in admissible special form at a special alternating diagram $D$. Let $D'$ be the diagram given by applying all the flypes indicated by the flype circles. Then $R'$ is the surface given by applying Seifert's algorithm to $D'$.
\end{lemma}
\begin{proof}
We prove this by induction on the number of flype circles. If there are no flype circles in the admissible special form then the result holds. 

Suppose there is at least one flype circle, and choose an innermost flype circle $\phi$. Then $D$ has the form shown in Figure \ref{coherencepic7}a, where $A$ contains no flype circles. 
The flype $\phi$ changes the diagram to that in Figure \ref{coherencepic7}b. 
This also gives $R'$ in admissible form. By induction, the result follows.
\begin{figure}[htb]
\centering
(a)
\input{pictexfiles/coherencepic7a}
(b)
\input{pictexfiles/coherencepic7b}
\caption{\label{coherencepic7}}
\end{figure}
\end{proof}

\subsection{Relating $\kc{D}$ and $\ms(L)$}

Our aim is to relate $\kc{D}$ to $\ms(L)$. We have a map $\mathcal{A}\colon\V(\kc{D})\to\V(\ms(L))$, as described in Section \ref{hssection}. To study the properties of this map, we will establish a local description at each vertex of $\kc{D}$.

\begin{proposition}\label{edgestoedgesprop}
Let $v_0\in\V(\kc{D})$. As above, $v_0$ corresponds to a prime, reduced, special alternating diagram $D_{v_0}$, and $R_0=\mathcal{A}(v_0)$ is given by applying Seifert's algorithm to $D_{v_0}$. Without loss of generality, we may assume $D=D_{v_0}$.

Let $v_1\in\V(\kc{D})$ with $\dist_{\kc{D}}(v_0,v_1)=1$. Then $\dist_{\ms(L)}(R_0,\mathcal{A}(v_1))\leq 1$.
\end{proposition}
\begin{proof}
Recall, for $R,R'\in\V(\ms(L))$, that $\dist_{\ms(L)}(R,R')\leq 1$ if and only if $R$ and $R'$ can be made disjoint.

Since $\dist_{\kc{D}}(v_0,v_1)=1$, vertex $v_1$ is obtained from $v_0$ by adding a sequence of distinct regions of $\theta(D)$.
Let $\Lambda$ be the union of these regions. Then $\partial\Lambda$ is split into its positive boundary $\partial^+\Lambda$ and its negative boundary $\partial^-\Lambda$. The change from $v_0$ to $v_1$ is achieved by subtracting 1 from the weight of all edges in $\partial^-\Lambda$ and adding one to the weight of all edges in $\partial^+\Lambda$.

Consider a $\theta$--graph in $\theta(D)$, with each edge labelled by the effect of adding $\Lambda$ (see, for example, Figure \ref{kdmslpic1}).
\begin{figure}[htbp]
\centering
\input{pictexfiles/kdmslpic1}
\caption{\label{kdmslpic1}}
\end{figure}
The total of these labels is 0, and the 1s and $-1$s must alternate. Pair each $-1$ with the 1 on its positive side. Each such pair defines a flype circle, whose crossing circle is the edge with label $-1$, and whose flype arc is the edge with the label 1. Note that this is always possible, because an edge with label $-1$ must have at least one crossing in $D$ corresponding to it.
Doing this for each $\theta$--graph gives a set of disjoint flype circles in $D$. Note also that these flype circles are coherent.

$\mathcal{A}(v_1)$ is given by applying these flypes to $D$, yielding a diagram $D_1$, and then applying Seifert's algorithm to $D_1$. By Lemma \ref{flypeslemma1}, $\mathcal{A}(v_1)$ is the surface described in admissible special form by this set of flype circles on $D$. Hence $\mathcal{A}(v_1)$ can be made disjoint from $R_0$.
\end{proof}

\begin{proposition}\label{kdmslprop2}
 Let $v_0$ and $R_0$ be as above. Let $R_1$ be a surface given in admissible special form at $D$. Then there exists $v_1\in\V(\kc{D})$ such that $\dist_{\kc{D}}(v_0,v_1)\leq 1$ and $\mathcal{A}(v_1)=R_1$.
\end{proposition}
\begin{proof}
Choose a set of flype circles showing that $R_1$ is in admissible special form. 
If there are no flype circles then $R_1=R_0$ and $v_1=v_0$. Assume there is at least one flype circle.

Each flype circle gives a weight of $-1$ to one edge of $\theta(D)$ and a weight of 1 to another edge. Modifying $v_0$ using these values gives another vector. Call this vector $v_1$. We aim to show that $v_1$ has the required properties.

Consider the set of flype circles in $\theta(D)$.
Note that a flype circle cannot run over an edge of $\theta(D)$ twice.
Suppose that two or more flype circles run over an edge $e\in\E(\theta(D))$. Choose two such flype circles $\phi,\phi'$ that are adjacent in $e$.
Since the flype circles are coherent, one must give $e$ a weighting of $1$ while the other gives it a weighting of $-1$. That is, $e$ is the flype crossing for $\phi$, say, and the flype arc for $\phi'$. By combining $\phi$ and $\phi'$ as shown in Figure \ref{kdmslpic2} to give a new flype circle $\phi''$, we can reduce the number of flype circles running over $e$ without changing the admissible form. Hence we may assume that at most one flype circle runs over any edge of $\theta(D)$.
\begin{figure}[htbp]
\centering
\input{pictexfiles/kdmslpic2}
\caption{\label{kdmslpic2}}
\end{figure}

Let $\Lambda\subseteq\sphere$ be the positive side of the flype circles, and $\lambda\subseteq\sphere$ the negative side. If $e$ receives a weighting of 1 then a flype arc runs along $e$, and $\Lambda$ lies on the negative side of $e$. If $e$ receives a weighting of $-1$ then $e$ corresponds to a flype crossing, and $\Lambda$ lies on the positive side of $e$. If $e$ receives a weighting of 0, no flype circle runs along it, and it lies in the interior of either $\Lambda$ or $\lambda$.

Each of $\Lambda$ and $\lambda$ is a non-empty union of regions of $\theta(D)$. We wish to show that those in $\Lambda$ can be ordered such that each partial composition of the corresponding maps is defined for $v_0$.
By symmetry, those in $\lambda$ can then also be suitably ordered so that the regions in $\Lambda$ followed by those in $\lambda$ define a simplex in $\kc{D}$.
Then $v_1\in\kc{D}$ and $\dist_{\kc{D}}(v_0,v_1)=1$.
It is sufficient to find a single region that can be added to $v_0$. Induction on the number of regions in $\Lambda$ then completes the proof.

Suppose no such region exists. Then each region $r$ in $\Lambda$ has a boundary edge $e(r)$ such that $r$ lies on the positive side of $e(r)$ and $w(e(r))=0$ (in $v_0$). This means $e(r)$ cannot lie on $\partial\Lambda$, so the region on the other side of $e(r)$ from $r$ is also in $\Lambda$. Construct a digraph $\mathcal{G}_{\Lambda}$ with an edge from $r$ across $e(r)$ for each $r$ in $\Lambda$. There are only finitely many regions in $\Lambda$, so $\mathcal{G}_{\Lambda}$ contains a simple closed curve $\rho$. Then there is a $\theta$--graph in $\theta(D)$ such that $\rho$ runs through every edge of the $\theta$--graph. Hence every edge of this $\theta$--graph has value 0 in $v_0$. This contradicts the construction of $\theta(D)$.
\end{proof}

\begin{lemma}\label{flaglemma}
The complex $\kc{D}$ is flag. That is, if the 1--skeleton of a simplex is contained in the complex then so is the simplex.
\end{lemma}
\begin{proof}
Let $v_0,\ldots,v_m$ be distinct vertices of $\kc{D}$ that are pairwise adjacent, for some $m\geq 2$. 
For $1\leq i\leq m$, there is a unique set $A_i$ of regions of $\theta(D)$ such that $v_i$ is obtained from $v_0$ by adding the regions in $A_i$ in some order. Let $A_0=\emptyset$. Then there is a partial order on the vertices $v_0,\ldots,v_m$ given by inclusion of the sets $A_0,\ldots,A_m$. 
We wish to show that $v_0,\ldots,v_m$ span a simplex in $\kc{D}$. From the proof of Proposition \ref{kdmslprop2} we see that it is sufficient to prove that this partial ordering of the vertices is actually a total order.

Suppose otherwise. Then there are two vertices that are not comparable. Without loss of generality, these are $v_1$ and $v_2$. Thus $A_1\setminus A_2$ and $A_2\setminus A_1$ are non-empty.
Let $\Lambda_1$ be the union of the regions in $A_1\setminus A_2$, and $\Lambda_2$ the union of the regions in $A_2\setminus A_1$. Then $\Lambda_1$ and $\Lambda_2$ have disjoint interiors. 

Suppose an edge $e$ of $\theta(D)$ lies on the boundary of $\Lambda_1$ and on the boundary of $\Lambda_2$. Then it lies on the positive boundary of one of $\Lambda_1,\Lambda_2$ and on the negative boundary of the other. Thus the coordinates of $v_1$ and $v_2$ corresponding to the edge $e$ differ by two. This contradicts that $v_1,v_2$ are adjacent. 

As $v_1$ is adjacent to $v_2$ in $\kc{D}$, there is a unique set $B$ of regions of $\theta(D)$ such that $v_2$ is obtained from $v_1$ by adding the regions in $B$ in some order. Let $\Lambda$ be the union of the regions in $B$. The boundary of $\Lambda$ is the union of the boundaries of $\Lambda_1$ and $\Lambda_2$, as these are the edges of $\theta(D)$ at which the coordinates of $v_1$ and $v_2$ differ. However, $\Lambda$ lies on the outside of $\partial \Lambda_1$ (that is, on the side away from $\Lambda_1$) and on the inside of $\partial \Lambda_2$, which is impossible. This gives the required contradiction.
\end{proof}

The following proposition says that $\ms(L)$ is also flag.
This means that each complex is defined by its 1--skeleton, so in relating them we may restrict our attention to comparing vertices and edges.

\begin{proposition}[\cite{MR1315011} Proposition 4.9]
Let $v_0,\ldots,v_n$ be distinct vertices of $\ms(L)$. If $v_i$ and $v_j$ span an edge in $\ms(L)$ whenever $i\neq j$, then $v_0,\ldots,v_n$ span an $n$--simplex in $\ms(L)$.
\end{proposition}

\section{Putting surfaces into admissible special form}\label{inductionsection}

Consider a surface in special form. Pick a white region of the diagram. At least one $P$--arc in this region will run across a crossing. Following Gabai (\cite{MR860665}), we can cut both the diagram and the surface along this crossing, giving a Seifert surface for a new link with fewer crossings than the original. The effect on the diagram and $P$--arcs is as shown in Figure \ref{regionpic1}. This gives the basis for an inductive argument.
\begin{figure}[htbp]
\centering
\input{pictexfiles/regionpic1}
\caption{\label{regionpic1}}
\end{figure}

Our inductive hypothesis in the proof of Proposition \ref{admissibleprop1} will relate to prime links. It is therefore useful to find a white region such that cutting along any crossing of that region will result in a prime link diagram.

\begin{theorem}[\cite{MR721450} Theorem 1]
Let $D$ be a reduced, alternating link diagram of a link $L$. Then $L$ is not prime if and only if this is visible in $D$. That is, if $L$ is not prime then there is a simple closed curve $\rho$ in $\sphere$ that is disjoint from the crossings of $D$ and meets the edges of $D$ exactly twice transversely such that there is at least one crossing on each side of $\rho$.
\end{theorem}

\begin{remark}
We will call a link diagram \textit{prime} if the link is prime. That is, we do not require a `prime link diagram' to be reduced.
\end{remark}

\begin{definition}
Say that such a region is \textit{cuttable}.
\end{definition}

\begin{lemma}\label{cutlemma}
If $D$ is a prime, reduced, special alternating diagram of a link $L$, then there is a cuttable white region of $D$.
\end{lemma}
\begin{proof}
Suppose otherwise. Choose a white region of $D$. Then this region has a crossing $c$ such that cutting along $c$ gives a connected sum. Since the original diagram was prime, $D$ has the form shown in Figure \ref{regionpic2}a, 
\begin{figure}[htbp]
\centering
(a)
\input{pictexfiles/regionpic2a}
(b)
\input{pictexfiles/regionpic2b}
\caption{\label{regionpic2}}
\end{figure}
where neither $A$ nor $B$ is trivial or consists of a single line of crossings surrounding black bigons. In particular, $A$ and $B$ must each contain a white region.
In this way, we can choose a simple closed curve $\rho_r$ for each white region $r$. We will allow two white regions to share one suitable curve.
Suppose, for regions $r,r'$, the curves $\rho_r,\rho_{r'}$ are distinct and cannot be isotoped to be disjoint. Then $D$ has the form shown in Figure \ref{regionpic2}b.
In this case, we could have chosen $\rho_r$ for both the regions $r$ and $r'$.
Thus the simple closed curves can be chosen to be disjoint. Choose a region $r$ such that $\rho_r$ is innermost. Then no white region $r'$ lies entirely inside $\rho_r$. This contradicts the choice of $\rho_r$.
\end{proof}

\subsection{Reducing a diagram}

In this section we consider some specific situations that will arise in the proof of Proposition \ref{admissibleprop1}.

Let $D, D^c$ be  special alternating link diagrams such that $D^c$ is obtained from $D$ by removing a crossing $c$ by a single type I Reidemeister move. Let $R$ be an incompressible Seifert surface given in special form at $D$. What changes need to be made to give $R$ in special form at $D^c$?

Any $P$--arcs that are distant from $c$ can be copied to $D^c$. We consider those close to $c$. 
The $P$--arcs around $c$ are as shown in Figure \ref{reducingpic1}a. Let $p$ be the point of $D^c$ corresponding to $c$ in $D$. 
\begin{figure}[htbp]
\centering
(a)
\input{pictexfiles/reducingpic1a}
(b)
\input{pictexfiles/reducingpic1b}
\caption{\label{reducingpic1}}
\end{figure}
One way to remove $c$ is to pull the undercrossing arc upwards into $\bbb{a}$ and then untwist. This leaves $R\cap\sphere$ as shown in Figure \ref{reducingpic1}b, giving a set of $P$--arcs on $D^c$, possibly with a single simple closed curve lying within the white region adjacent to $p$. Does this give a special form for $R$? One of the discs of $R\cap \bbb{b}$ has been moved, with a subdisc $S$ along part of its boundary being pulled up into $\bbb{a}$. Clearly this still leaves a disc in $\bbb{b}$. The disc $S$ is glued onto discs of $R\cap \bbb{a}$ along two disjoint subarcs of its boundary. If these two arcs are glued to distinct discs in $\bbb{a}$, this again gives a disc, and the $P$--arcs give a special form of $R$.
Alternatively, the two arcs may be glued to the same disc, in which case an annulus $T$ is formed. The core curve of $T$ bounds a disc in $\bbb{a}$ disjoint from $R$. As $R$ is incompressible, one component of $\partial T$ bounds a disc in $R$ disjoint from the other. This component cannot run along the link, so must be a simple closed curve $\rho$ of $R\cap\sphere$.
In this case, the $P$--arcs around $c$ in $D$ must be as shown in Figure \ref{reducingpic2}. The $P$--arcs in $D^c$, without $\rho$, give a special form for $R$.
\begin{figure}[htbp]
\centering
\input{pictexfiles/reducingpic2}
\caption{\label{reducingpic2}}
\end{figure}

We could instead have removed $c$ by first pushing the overcrossing arc downwards into $\bbb{b}$. Similar reasoning applies in this case.
If no annulus is formed in either case then around $c$ the simple closed curves bounding the discs of $R$ in $\bbb{a}$ and $\bbb{b}$, which are made up of $P$--arcs and arcs of $D$, must be connected as shown in Figure \ref{reducingpic3}.
\begin{figure}[htbp]
\centering
\input{pictexfiles/reducingpic3}
\caption{\label{reducingpic3}}
\end{figure}

\begin{definition}
Call the change of $D$ to $D^c$ given by pushing $c$ into $\bbb{a}$ and untwisting an \textit{$a$--reduction at $c$}, and say that we \textit{$a$--reduce} $D$. Similarly define a \textit{$b$--reduction} and \textit{$b$--reducing}.
\end{definition}

\begin{remark}\label{untwistremark}
If there is a $P$--arc across $c$ in $D$ then $a$--reducing $D$ at $c$ and $b$--reducing it give the same result, and do not affect any of the other $P$--arcs.
\end{remark}

For the rest of this section we assume $D^c$ is the $a$--reduction of $D$.

\begin{lemma}\label{oneuntwistlemma}
Suppose $D$ and $D^c$ are prime, and $D^c$ is reduced.
Suppose further that the special form for $R$ at $D^c$ is admissible (and hence $R$ is minimal genus). 
Then there is a $P$--arc in $D$ across $c$.
\end{lemma}
\begin{proof}
Consider the flype circles in the admissible special form. Let $D^c_0$ be the diagram obtained from $D^c$ by performing all the flypes so as to leave $p$ fixed. By Lemma \ref{flypeslemma1}, the surface given by the $P$--arcs on $D^c_0$ is $R$, and in particular is minimal genus.

This process of changing the diagram by a flype $\phi$ moves a $P$--arc only if it lies strictly inside $\phi$ on the side that is moved. Since $p$ was not moved by any flype, the $P$--arcs in $D^c$ that came from those adjacent to $c$ in $D$ were also not moved. This means we can perform the same changes to the diagram $D$ with its $P$--arcs as were made to $D^c$, to give a diagram $D_0$. These changes will correspond to performing flypes on $D$, but these flypes may not come from flype circles in the special form on $D$. However, we see that $D^c_0$ is obtained from $D_0$ by untwisting $c$, and the $P$--arcs in $D^c_0$ are those given by $a$--reducing $D_0$ at $c$. Hence we may assume that the admissible special form of $R$ at $D^c$ has no flype circles.
 
If either $a$--reducing or $b$--reducing $D$ at $c$ creates an annulus as described above, there is a $P$--arc across $c$ in $D$. Hence we may assume no annulus is formed. We use this to build up a picture of part of $D^c$, and so derive a contradiction.

From above (Figure \ref{reducingpic3}), in $\bbb{a}$ the $P$--arcs on $D^c$ near $p$ are connected as shown in Figure \ref{reducingpic4}.
\begin{figure}[htb]
\centering
\input{pictexfiles/reducingpic4}
\caption{\label{reducingpic4}}
\end{figure}
Since there is no $P$--arc across $c$ in $D$, $\sigma_1$ cannot consist entirely of $P$--arcs. Then, since there are no flype circles in the admissible special form on $D^c$, $\rho_1$ and $\rho_2$ each run across a single crossing.  From this we have Figure \ref{reducingpic5}.
\begin{figure}[htb]
\centering
\input{pictexfiles/reducingpic5}
\caption{\label{reducingpic5}}
\end{figure}
There is a $P$--arc across $c_1$, so a neighbourhood of $c_1$ is as shown in Figure \ref{reducingpic6}a.
\begin{figure}[htb]
\centering
(a)
\input{pictexfiles/reducingpic6a}
(b)
\input{pictexfiles/reducingpic6b}
\caption{\label{reducingpic6}}
\end{figure}
Consider the next crossing $c_2$ around the black region $r$ that $\sigma_2$ meets. Again, there is a $P$--arc across $c_2$, so we have the set-up shown in Figure \ref{reducingpic6}b.
Inductively we see that the arc $\sigma_2$ will return to $\partial r$ after every crossing. Thus $D^c$ has the structure shown in Figure \ref{reducingpic7}.
\begin{figure}[htbp]
\centering
\input{pictexfiles/reducingpic7}
\caption{\label{reducingpic7}}
\end{figure}
Because $D^c$ is reduced and alternating, this contradicts that it is prime.
\end{proof}

\begin{lemma}\label{twountwistlemma}
Suppose there is a reduced diagram $D^{c,c'}$ given by $a$--reducing $D^c$ at a crossing $c'$, where $p$ does not lie on the edge of $D^c$ connecting $c'$ to itself. 
Suppose further that all the diagrams are prime and the special form of $R$ at $D^{c,c'}$ is admissible.
Then in $D$ there are $P$--arcs across each of $c$ and $c'$.
\end{lemma}
\begin{proof}
By Lemma \ref{oneuntwistlemma} there is a $P$--arc across $c'$ in $D^c$.
Suppose there is no $P$--arc across $c'$ in $D$. Then $c$ and $c'$ lie on the same white region of $D$, and the $P$--arcs around them are connected in one of the patterns shown in Figure \ref{reducingpic8}.
\begin{figure}[htb]
\centering
(a)
\input{pictexfiles/reducingpic8a}
(b)
\input{pictexfiles/reducingpic8b}
\caption{\label{reducingpic8}}
\end{figure}
In the first case, the $P$--arcs around $c'$ are not connected as in Figure \ref{reducingpic3}. Therefore the situation in Figure \ref{reducingpic8}b occurs.

Consider the set of flype circles in $D^{c,c'}$ that show that the special form of $R$ is admissible. None of these can separate $c$ and $c'$ in $D$. Thus, as in the proof of Lemma \ref{oneuntwistlemma}, we may assume that there are no flype circles needed and every crossing in $D$ has a $P$--arc across it. Further following the proof of Lemma \ref{oneuntwistlemma} shows that $D$ has the structure shown in Figure \ref{reducingpic9}, contradicting that it is prime.
\begin{figure}[htb]
\centering
\input{pictexfiles/reducingpic9}
\caption{\label{reducingpic9}}
\end{figure}

Thus there is a $P$--arc across $c'$ in $D$. This means that untwisting $c'$ has no impact on the rest of the picture, so Lemma \ref{oneuntwistlemma} gives that there is a $P$--arc across $c$ in $D$.
\end{proof}

\begin{lemma}\label{multiuntwistlemma}
Suppose $p$ lies on an edge of $D^c$ connecting a crossing $c'$ to itself, so $c'$ can be removed by a type I Reidemeister move.
Suppose the special form of $R$ at $D^c$ has a $P$--arc across $c'$. Then there is a $P$--arc across $c$ in $D$.
\end{lemma}
\begin{proof}
Suppose otherwise. From above (Figure \ref{reducingpic3}) the arcs of $R\cap \bbb{a}$ on $D$ must connect as in Figure \ref{reducingpic10}a, where $\sigma$ does not consist entirely of $P$--arcs. This contradicts that there is a $P$--arc across $c'$ in $D^c$ (see Figure \ref{reducingpic10}b).
\begin{figure}[htb]
\centering
(a)
\input{pictexfiles/reducingpic10a}
(b)
\input{pictexfiles/reducingpic10b}
\caption{\label{reducingpic10}}
\end{figure}
\end{proof}

\subsection{Adjacent surfaces can be put into admissible special form}

\begin{proposition}\label{admissibleprop1}
Let $D$ be a prime, reduced, special alternating diagram of a link $L$. Let $R$ be the (minimal genus) Seifert surface for $L$ given by applying Seifert's algorithm to $D$. Let $R'$ be a minimal genus Seifert surface for $D$ disjoint from $R$, given in special form. Then this special form is admissible.
\end{proposition}
\begin{proof}
First suppose $D$ has no crossings. Then $L$ is the unknot, $R'=R$ and the special form of $R'$ is admissible with no $P$--arcs or flype circles.

Now suppose the result holds for any diagram with at most $n-1$ crossings, and that $D$ has $n$ crossings. 
By Lemma \ref{cutlemma}, there is a cuttable region $r$ of $D$. Choose a crossing $c$ of $r$ with a $P$--arc across it. Cut along this $P$--arc and $c$ as in Figure \ref{regionpic1} to give a new diagram $D_c$ and a new minimal genus Seifert surface $R'_c$ in special form at $D_c$. Since $r$ is cuttable, $D_c$ is prime. It is also special and alternating, and has $n-1$ crossings. However, it may not be reduced.

Suppose $D_c$ is reduced. Then the inductive hypothesis holds, so the special form of $R'_c$ is admissible. Take a flype circle $\phi$ in $D_c$, and consider $\phi$ in $D$. The flype arc of $\phi$ is a $P$--arc in $D_c$, so also is in $D$. This means that $c$ and the $P$--arc across it lie on one side of $\phi$, so $\phi$ is a flype circle in $D$. Thus the flype circles in $D_c$ give a set of flype circles in $D$. These flype circles inherit coherence from $D_c$.
Suppose $c$ does not lie between a crossing $c'$ and the $P$--arc that crosses it in $D_c$. Then each crossing in $D$ that is not a flype crossing or $c$ has a $P$--arc across it on the required side. By counting the number of endpoints of $P$--arcs in each white region of $D$ we find that the $P$--arc across $c$ also lies on the required side. Thus the special form of $R'$ is admissible.
Instead suppose $c$ does lie between a crossing $c'$ and the $P$--arc that crosses it in $D_c$. We see that $c$ and $c'$ are as show in Figure \ref{inductionpic1}. Again, the special form of $R'$ at $D$ is admissible.
\begin{figure}[htbp]
\centering
\input{pictexfiles/inductionpic1}
\caption{\label{inductionpic1}}
\end{figure}

Now suppose $D_c$ is not reduced. Take a simple closed curve $\rho$ that shows that $D_c$ is not reduced. That is, $\rho$ passes through a single crossing of $D_c$ and otherwise lies in a white regions of $D_c$. Since $D$ is reduced, $\rho$ must pass through $c$ in $D$, giving $D$ the form shown in Figure \ref{inductionpic2}.
\begin{figure}[htbp]
\centering
\input{pictexfiles/inductionpic2}
\caption{\label{inductionpic2}}
\end{figure}
As $D$ is reduced and $D_c$ is prime, either $A$ or $B$ 
consists of a single line of crossings. Thus cutting along $c$ leaves one or two lines of nugatory crossings joined by black bigons in a single white region of what is otherwise a reduced diagram (for example as shown in Figure \ref{inductionpic3}). 
\begin{figure}[htbp]
\centering
\input{pictexfiles/inductionpic3}
\caption{\label{inductionpic3}}
\end{figure}

Repeatedly $a$--reduce the diagram until only one crossing from each of these lines remain, giving a new diagram $D_1$. Let $D_0$ be the diagram given by $a$--reducing the final one or two nugatory crossings of $D_1$. Then, by the inductive hypothesis, the special form of $R'$ at $D_0$ is admissible, so the hypotheses of either Lemma \ref{oneuntwistlemma} or Lemma \ref{twountwistlemma} hold in relation to $D_1$. One of these lemmas followed by repeated application of Lemma \ref{multiuntwistlemma} gives that the special form of $R'_c$ at $D_c$ has a $P$--arc across every nugatory crossing. Thus $a$--reducing $D_c$ does not affect the other $P$--arcs, and every flype circle in $D_0$ is a flype circle in $D$. 

Let $c'$ be a crossing of $D_0$. If $c'$ is a flype crossing in $D_0$ and in $D$ 
then the flype arc of the flype circle is a $P$--arc in $D$. 
Suppose $c'$ is not a flype crossing. Then there is a $P$--arc $\rho$ across it on the required side in $D_0$. Suppose that $\rho$ is not across $c'$ in $D$. Then at least one of the nugatory crossings of $D_1$ lies between $c'$ and $\rho$. This then means that in fact $c$ and all of the nugatory crossings of $D_c$ lie between $c'$ and $\rho$. Hence there is at most one crossing in $D_0$ for which this occurs.

If no such crossing exists, we may again check that, for each crossing of $D$ that is not in $D_0$, the $P$--arc across it lies on the required side. 
Suppose there is such a crossing $c_{\phi}$. Then $c_{\phi}$ and the $P$--arc $\rho_{\phi}$ across it in $D_0$ form a new flype circle $\phi$ in $D$. Without loss of generality, $c$ lies on the positive side of $\phi$. It is then easily checked that every crossing on the positive side of $\phi$ in $D$ has a $P$--arc across it on the negative side. Since we know that $\rho_{\phi}$ lies on the positive side of $c_\phi$ in $D_0$, the crossing $c_{\phi}$ (and hence also the flype circle $\phi$) lies on the negative side of the flype circles in $D_0$. Hence these together with $\phi$ give a set of flype circles in $D$ proving that the special form of $R'$ at $D$ is admissible.
\end{proof}

\begin{corollary}[\cite{MR1664976} Theorem 1.1]
Let $L$ be a prime, special alternating link with a reduced, special alternating digram $D$. Let $R$ be a minimal genus Seifert surface for $L$. Then $R$ is given by doing a finite sequence of flypes on $D$ and then applying Seifert's algorithm to the resulting diagram.
\end{corollary}
\begin{proof}
This is a combination of Proposition \ref{admissibleprop1}, Lemma \ref{flypeslemma1} and the fact that $\ms(L)$ is connected.
\end{proof}

\begin{corollary}
Every minimal genus Seifert surface for $L$ is connected.
\end{corollary}

\section{Distinguishing surfaces in special form}\label{disjointnesssection}

Our current aim is to prove Theorem \ref{isomorphismtheorem}.
To do so, we wish to show that the map $\mathcal{A}\colon\V(\kc{D})\to\V(\ms(L))$ extends to an isomorphism of the simplicial complexes. 
By Proposition \ref{edgestoedgesprop} we know that $\mathcal{A}$ can be extended to a map on the edges of the complexes. Since $\ms(L)$ is flag, this gives a simplicial map $\mathcal{A}\colon\kc{D}\to\ms(L)$. 
We will show that $\mathcal{A}$ is a local isomorphism at each point, and so is a covering map. As $\ms(L)$ is simply connected, $\mathcal{A}$ is therefore an isomorphism.

Because both $\ms(L)$ and $\kc{D}$ are flag, we need only show that $\mathcal{A}$ acts as required on the 1--skeleton of $\kc{D}$. That is, we will consider the 1--skeleton of the closure of the star of a vertex $v$ in $\kc{D}$, and show that this is mapped bijectively under $\mathcal{A}$ to the corresponding subgraph of $\ms(L)$. From Proposition \ref{admissibleprop1}, together with Proposition \ref{kdmslprop2}, we already know that this is a surjection. In this section we will show that the distances between the vertices are not reduced. This will complete the proof of Theorem \ref{isomorphismtheorem}.

\subsection{Product regions and product discs}

\begin{definition}[\cite{2010arXiv1004.4168P} Section 3]
Let $M$ be a connected 3--manifold, and let $S,S'$ be (possibly disconnected) surfaces properly embedded in $M$.
Call $S$ and $S'$ \textit{almost transverse} if, given a component $S_0$ of $S$ and a component $S'_0$ of $S'$, they either coincide or intersect transversely.
Say $S$ and $S'$ \textit{bound a product region} if the following holds. There is a compact surface $T$, a finite collection $\rho\subseteq \partial T$ of arcs and simple closed curves and a map of $N=(T\times\intvl)/\sim$ into $M$ that is an embedding on the interior of $N$ and has the following properties. 
\begin{itemize}
 \item $T\times\{0\}=S\cap N$ and $T\times\{1\}=S'\cap N$.
 \item $\partial N\setminus (T\times\partial\intvl)\subseteq\partial M$.
\end{itemize}
Here $\sim$ collapses $\{x\}\times\intvl$ to a point for each $x\in\rho$.
Say $S$ and $S'$ have \textit{simplified intersection} if they do not bound a product region.
\end{definition}

\begin{proposition}[\cite{MR1315011} Proposition 4.8]\label{productregionsprop}
Let $M$ be a Haken manifold with incompressible boundary.
Suppose $S,S'$ are incompressible, $\partial$--incompressible surfaces properly embedded in $M$ in general position. Suppose further either that $S\cap S'\neq\emptyset$ but $S$ can be isotoped to be disjoint from $S'$, or that $S,S'$ are isotopic. Then $S,S'$ bound a product region.
\end{proposition}

\begin{remark}
If $M=\Sphere\setminus\nhd(L)$ for a non-split link $L$ other than the unknot then $M$ is Haken and $\partial$--irreducible. Furthermore, if $S,S'$ are minimal genus Seifert surfaces for $L$ then they are properly embedded, incompressible and $\partial$--incompressible. 
\end{remark}

\begin{definition}
A \textit{sutured manifold} $(M,s)$ is a compact, orientable 3--manifold $M$, together with a finite set $s$ of disjoint simple closed curves on $\partial M$, called the \textit{sutures}. The sutures divide $\partial M$ into two (possibly disconnected) compact, oriented surfaces $S_+(M)$ and $S_-(M)$ such that $S_+(M)\cap S_-(M)=s$ and, if $\rho$ is a suture, $S_+(M)$ and $S_-(M)$ meet at $\rho$ with opposite orientations. 
\end{definition}

\begin{definition}
A \textit{product sutured manifold} is a sutured manifold $(M,s)$ that is homeomorphic to $S_+(M)\times[-1,1]$ with $s=\partial S_+(M)\times\{0\}$.
\end{definition}

\begin{remark}
A product region is a product sutured manifold, with sutures along the core curves of $(S\cap S')\cup \partial M$, although the orientations of the boundary of the sutured manifold may not agree with those of $S$ and $S'$.
\end{remark}

\begin{definition}
A disc $T$ properly embedded in a sutured manifold $(M, s)$ is a \textit{product disc} if $\partial T$ meets $s$ at exactly two points, where it crosses $s$ transversely.
\end{definition}

\begin{definition}
Let $(M, s)$ be a sutured manifold that contains a product disc $T$. Let $\rho$ be a simple arc on $T$ joining the two points of $\partial T \cap s$ and let $T \times [-1,1]$ be a product neighbourhood of $T$ in $M$. The sutured manifold $(M', s')$ obtained from $(M, s)$ by a \textit{product disc decomposition} along $T$ has $M' = M \setminus ( T \times(-1,1))$ and $s' = (s \cap M') \cup (\rho \times\{\pm 1\})$.
\end{definition}

\begin{remark}\label{productremark}
$M'$ is a product sutured manifold if and only if $M$ is.
\end{remark}

\begin{definition}
Say a product disc is \textit{inessential} if it is separating and the disc decomposition along it creates a component that is a 3--ball with a single suture. 
Otherwise, it is \textit{essential}. 
\end{definition}

\begin{remark}
If a product disc $T$ in a sutured manifold $M$ is inessential then $\partial T\cap S^+(M)$ is inessential in $S^+(M)$, and the same is true for $S^-(M)$.
\end{remark}

\begin{remark}\label{productdiscremark}
Suppose $T'$ is an incompressible surface properly embedded in $M$ with $\partial T'=s$. By an isotopy we can ensure that $T'\cap (T\times[-1,1])=\rho\times[-1,1]$. Then $T'\setminus(T\times(-1,1))$ is a surface properly embedded in $M'$ with boundary $s'$. 
\end{remark}

\begin{lemma}\label{productdiscslemma}
Let $(M,s)$ be a sutured manifold, and let $\rho$ be an essential arc properly embedded in $S^+(M)$. Let $T$ be a product disc in $M$ such that $\partial T$ meets $\rho$ exactly once and such that $\rho$ cannot be isotoped to be disjoint from $\partial T$. Let $(M',s')$ be the result of the product disc decomposition along $T$, and let $\rho_1,\rho_2$ be the two parts of $\rho$ in $M'$.
Suppose there is an essential product disc $S$ in $M$ with $\rho\subset\partial S$. Then there are essential product discs $S_1,S_2$ in $M'$ with $\rho_1\subset\partial S_1,\rho_2\subset\partial S_2$.
\end{lemma}
\begin{proof}
Notice that we are free to change $S$, provided we always have $\rho\subset\partial S$. 
A simple closed curve in $T\cap S$ that is innermost in $T$ bounds a disc both in $T$ and in $S$. Replacing the subdisc of $S$ with the subdisc of $T$ reduces $|T\cap S|$. Thus we can remove all simple closed curves from $T\cap S$. 
Exactly one arc of $T\cap S$ has an endpoint on $S^+(M)$, at the point where $\partial T$ crosses $\rho$, and the other endpoint of this arc lies on $S^-(M)$. Suppose there is at least one other arc of $T\cap S$. Choose such an arc that is outermost in $T$. This cuts off a subdisc of $T$ disjoint from $S^+(M)$ with interior disjoint from $S$. It also cuts off a subdisc of $S$ that is disjoint from $S^+(M)$. Replacing the subdisc of $S$ with the subdisc of $T$ again reduces $|T\cap S|$.
We may assume, therefore, that $T\cap S$ is a single arc, running from $S^+(M)$ to $S^-(M)$. Since $S\cap S^+(M)=\rho$ has not changed, $S$ is still essential.

Now the disc decomposition along $T$ cuts $S$ into two product discs $S_1, S_2$. Suppose $S_1$ is inessential. Then $\rho_1=\partial S_1\cap S^+(M')$ is inessential in $S^+(M')$. Let $T_1$ be a disc in $\partial M'$ between $\rho_1$ and $s'$. Then $T_1^*=T_1\cap\partial M$ is a disc. This disc defines an isotopy of $\rho\cap T_1^*$ in $\partial M$ that makes $\rho$ disjoint from $\partial T$. This contradicts that no such isotopy exists. Thus $S_1$ is essential. Similarly, so in $S_2$.
\end{proof}

\begin{definition}
For a sutured manifold $(M,s)$ embedded in $\Sphere$, the \textit{complementary sutured manifold} $(M',s')$ is defined by $M'=\Sphere\setminus \Int_{\Sphere}(M)$ and $s'=s$.

By the \textit{complementary sutured manifold to a Seifert surface $R$} we mean the complementary sutured manifold to the product sutured manifold given by a product neighbourhood of $R$.
\end{definition}

\begin{remark}\label{fibredremark}
A link $L$ with a Seifert surface $R$ is fibred with fibre $R$ if and only if the complementary sutured manifold to $R$ is a product sutured manifold.
\end{remark}

\begin{lemma}\label{hopfbandlemma}
Let $D$ be a special alternating diagram, and let $r$ be a white bigon region of $D$. Then $r$ defines a product disc in the complementary sutured manifold to the Seifert surface $R$ given by applying Seifert's algorithm to $D$. The effect of the product disc decomposition is to remove the region $r$, replacing the two crossings of $r$ with a single crossing. Such a change on $D$ has the effect of deplumbing a Hopf band from $R$ (see Figure \ref{productregionspic1}).
\begin{figure}[htbp]
\centering
\input{pictexfiles/productregionspic1}
\caption{\label{productregionspic1}}
\end{figure}
\end{lemma}

\begin{proposition}[\cite{MR2131376} Proposition 1.4]\label{samelinkprop}
Let $R$ be a connected minimal genus Seifert surface for a link $L$. Let $R'$ be the surface obtained from $R$ by a product disc decomposition, and let $L'=\partial R'$. Then the link of the vertex $R$ in $\ms(L)$ is isomorphic to the link of $R'$ in $\ms(L')$. 
Furthermore, the isomorphism is induced by the procedure described in Remark \ref{productdiscremark}.
\end{proposition}

\begin{definition}
Given a special alternating link diagram $D$, let $G(D)$ be the planar graph with a vertex in each white region of $D$ and an edge through each crossing.
\end{definition}

\begin{theorem}[see \cite{MR2001624} Theorem 3(5) and \cite{2011arXiv1101.1412B} Theorem 5.1]\label{fibredgraphthm}
Let $D$ be a special alternating diagram of a prime link $L$. Then $L$ is fibred if and only if the graph $G(D)$ can be reduced to a single vertex by a sequence of the following moves.
\begin{itemize}
\item Delete a loop (that is, an edge with both endpoints at the same vertex).
\item Contract an edge, one endpoint of which is at a vertex with valence 2.
\end{itemize}
Informally, $G(D)$ satisfies this condition exactly when it is a `tree of loops'.
\end{theorem}

\begin{lemma}
Every vertex of $G(D)$ has even valency.
If $D$ is a reduced diagram of a prime link then $G(D)$ has no cut vertices or loops.
\end{lemma}

\subsection{Adjacent surfaces are distinct}

\begin{definition}
Given a special alternating diagram $D$, let $R$ be the Seifert surface given by applying Seifert's algorithm. Recall that $R$ is oriented. Divide the black regions of $D$ into \textit{$a$--regions} and \textit{$b$--regions} according to whether the normal to $R$ in a given region points into $\bbb{a}$ or into $\bbb{b}$.
\end{definition}

\begin{proposition}\label{disjointnessprop1}
Let $D$ be a reduced, special alternating diagram for a prime link $L$. Let $v,v'\in\V(\kc{D})$ such that $\dist_{\kc{D}}(v,v')=1$ and $v$ is given by $D$. Then $R=\mathcal{A}(v)$ is given by applying Seifert's algorithm to $D$. Let $R'=\mathcal{A}(v')$. Then $\dist_{\ms(L)}(R,R')=1$.
\end{proposition}
\begin{proof}
Since $\dist_{\kc{D}}(v,v')=1$, the proof of Proposition \ref{edgestoedgesprop} gives an admissible special form for $R'$ at $D$. This special form has at least one flype circle. Choose a set of flype circles that minimises the number of flype circles (as in the proof of Proposition \ref{kdmslprop2}). Together $R$ and $R'$ cut $\Sphere\setminus\nhd(L)$ into two sutured manifolds. Fix an $a$--region $r$ of $D$. Call the two manifolds $M^+$ and $M^-$, where $M^+$ lies above $r$ and $M^-$ lies below it. We will show that neither of $M^+$ and $M^-$ is a product sutured manifold. By symmetry, it is sufficient to prove this for $M^+$. The result will then follow by Proposition \ref{productregionsprop}.

Consider a flype circle $\phi$. The arcs shown in bold in Figure \ref{disjointnesspic1} 
\begin{figure}[htbp]
\centering
\input{pictexfiles/disjointnesspic1}
\caption{\label{disjointnesspic1}}
\end{figure}
form part of the boundary of a single disc of $R'\cap \bbb{a}$. To see this, note that only two overcrossing arcs of $L$ cross the boundary of $A$. Hence one of $\rho_1$ and $\rho_2$ defines a product disc in $M^+$ contained in $\bbb{a}$. Assume, without loss of generality, this is $\rho_1$. Then the product disc in $M^+$ defined by $\rho_3$ is contained in $\bbb{b}$.

Up to isotopy, $M^+\cap \bbb{a}$ and $M^+\cap \bbb{b}$ are defined by where they meet $\sphere$. 
We now consider where $M^+$ lies relative to $D$ and the special form of $R'$. As $M^+$ always meets the same side of $R$, it lies above all $a$--regions of $D$, and below all $b$--regions. In the white regions, $M^+$ lies on exactly one side of each $P$--arc. Let $\rho$ be a $P$--arc. First suppose $\rho$ is the flype arc of a flype circle $\phi$. If $\phi$ and $\rho$ are as in Figure \ref{disjointnesspic2}, 
\begin{figure}[htbp]
\centering
\input{pictexfiles/disjointnesspic2}
\caption{\label{disjointnesspic2}}
\end{figure}
where the black region $r'$ shown is an $a$--region, then $M^+$ lies on the positive side of $\phi$, as marked $*$.
Instead suppose that $\rho$ lies across a crossing $c$. If $c$ lies on the negative side of the flype circles, $\rho$ lies on the positive side of $c$ and $M^+$ lies between $\rho$ and $c$. Otherwise, $\rho$ lies on the positive side of the flype circles. In this case, $M^+$ does not lie between $c$ and $\rho$. It lies on the far side of $c$ from $\rho$, and on the far side of $\rho$ from $c$.

From this we see that on the negative side of the flype circles $R$ and $R'$ are parallel, with $M^+$ as the product region between them. This product structure extends to the product discs described above where $M^+$ meets each flype circle. Thus we only need to consider the pieces of $M^+$ left by removing these product regions.
Let $\Phi$ be the union of the flype circles. The remaining pieces of $M^+$ correspond to the components of $\sphere\setminus\Phi$ that lie on the positive side of the flype circles. Let $\Lambda$ be one such component, and $M^+_{\Lambda}$ the corresponding piece of $M^+$. Create a new diagram $D_{\Lambda}$ from $D$ by changing it as shown in Figure \ref{disjointnesspic3} 
\begin{figure}[htbp]
\centering
\input{pictexfiles/disjointnesspic3}
\caption{\label{disjointnesspic3}}
\end{figure}
along each flype circle that bounds $\Lambda$. Let $R_{\Lambda}$ be the surface given by applying Seifert's algorithm to $D_{\Lambda}$, and $R'_{\Lambda}$ that given by the $P$--arcs. Since the special form of $R'_{\Lambda}$ has no flype circles, $R_{\Lambda}$ and $R'_{\Lambda}$ are parallel through a product sutured manifold $N$. The complementary sutured manifold to $N$ is isotopic to $M^+_{\Lambda}$. Hence $M^+$ is a product sutured manifold if and only if every diagram $D_{\Lambda'}$ constructed in this way is of a fibred link with fibre given by the surface $R_{\Lambda'}$. Again by symmetry, we need only consider $D_{\Lambda}$.

Consider $G(D)$ and $G(D_{\Lambda})$, and a fixed flype circle $\phi_{\Lambda}$ that bounds $\Lambda$. As can be seen from in Figure \ref{disjointnesspic4}, 
\begin{figure}[htbp]
\centering
\input{pictexfiles/disjointnesspic4}
\caption{\label{disjointnesspic4}}
\end{figure}
$G(D_{\Lambda})$ is obtained from $G(D)$ by collapsing everything on the negative side of the flype circles. There may be other flype circles in $A$ that are collapsed to give $A'$. Assume $D_{\Lambda}$ is fibred. 
Then by Theorem \ref{fibredgraphthm} the edge $e_{\Lambda}$ forms part of a subdivision of a loop.

Suppose that $G(D_{\Lambda})$ is not homeomorphic to $\crcle$. We can then divide up $D_{\Lambda}$ and $G(D_{\Lambda})$ as shown in Figure \ref{disjointnesspic5}, 
\begin{figure}[htbp]
\centering
\input{pictexfiles/disjointnesspic5}
\caption{\label{disjointnesspic5}}
\end{figure}
where at least three edges of $G(D_{\Lambda})$ run from each of $v_1$ and $v_2$ into $A''$. Because of the `tree-like' structure of $G(D_{\Lambda})$, either $v_1$ is a cut vertex or there is a loop attached to $v_1$. The collapsing of $G(D)$ to $G(D_{\Lambda})$ does not create either of these. Hence $v_1$ is also a cut vertex or has a loop attached to it in $G(D)$. This contradicts that $D$ is prime and reduced. 

Therefore, $G(D_{\Lambda})$ is homeomorphic to $\crcle$. Since the flype $\phi_{\Lambda}$ is essential, at least one vertex of $G(D_{\Lambda})$ in $A'$ comes from collapsing a flype circle in $A$. By `uncollapsing' all such vertices in $G(D_{\Lambda})$ we see that $G(D)$ decomposes as shown in Figure \ref{disjointnesspic6}. 
We can then combine two flype circles as in Figure \ref{kdmslpic2}, contradicting our choice of flype circles. Thus $D_{\Lambda}$ is not fibred.
\begin{figure}[htbp]
\centering
\input{pictexfiles/disjointnesspic6}
\caption{\label{disjointnesspic6}}
\end{figure}
\end{proof}

\begin{proposition}\label{disjointnessprop2}
Let $D$ be a reduced, special alternating diagram for a prime link $L$, and let $v_0\in\V(\kc{D})$ be given by $D$.
Let $v_{-1},v_1\in\V(\kc{D})$ such that $\dist_{\kc{D}}(v_{-1},v_0)=\dist_{\kc{D}}(v_0,v_1)=1$ and $\dist_{\kc{D}}(v_{-1},v_1)=2$.  Let $R_i=\mathcal{A}(v_i)$. Then $\dist_{\ms(L)}(R_{-1},R_1)=2$.
\end{proposition}
\begin{remark}
This proof requires careful consideration of the relative positions of the Seifert surfaces in $\Sphere$ and their interaction with the projection plane $\sphere$. As such, it is more suited to being drawn in a specific case than described in words for a general link. Therefore, a worked example is given in Appendix \ref{firstappendix}.
\end{remark}
\begin{proof}
Let $\nhd(L)$ denote the regular neighbourhood of $L$.
Following the proof of Proposition \ref{edgestoedgesprop}, construct an admissible special form at $D$ for $R_{-1}$ and for $R_1$. These special forms will each have at least one flype circle. 
Position each flype circle with the distance from the positive side of the flype crossing to the negative side of the flype arc as small as possible.
Let $\Phi$ denote the set of flype circles for $R_{-1}$ and $R_1$, where we do not require that each passes through the vertices of $\theta(D)$ and instead make them disjoint.
For $i=\pm 1$, let $\Lambda_i$ be the positive side of the flype circles for $R_i$, and $\lambda_i$ the negative side. Note that each of these is a union of regions of $\theta(D)$.

Suppose that at least one region of $\theta(D)$ lies in $\Lambda_{-1}\cap\Lambda_1$.
As in the proof of Proposition \ref{kdmslprop2}, 
there is at least one region $r$ of $\theta(D)$ in $\Lambda_{-1}\cap\Lambda_1$ such that $r^{\theta}$ is defined at $D$.
Adding $r$ to $v_0$ gives a new vertex $v'_0$ in $\kc{D}$ with $\dist_{\kc{D}}(v_{-1},v'_0)\leq 1$ and $\dist_{\kc{D}}(v'_0,v_1)\leq 1$. Thus $\dist_{\kc{D}}(v_{-1},v'_0)=\dist_{\kc{D}}(v'_0,v_1)= 1$ since $\dist_{\kc{D}}(v_{-1},v_1)=2$.
Hence, without loss of generality, we may assume $\Lambda_{-1}\cap\Lambda_1$ does not contain any region of $\theta(D)$. In particular this means that no two flype arcs lie on any one edge of $\theta(D)$, and neither do two flype crossings.

Let $\lambda$ be a component of $\lambda_{-1}$.
Create a new diagram $D_{\lambda}$ by, for every flype circle $\phi$ of $R_{-1}$ on the boundary of $\lambda$, changing $D$ as shown in Figure \ref{disjointnesspic7}a. All $P$--arcs in $\lambda$ can be copied to $D_{\lambda}$.
\begin{figure}[htbp]
\centering
(a)
\input{pictexfiles/disjointnesspic7a}
(b)
\input{pictexfiles/disjointnesspic7b}
\caption{\label{disjointnesspic7}}
\end{figure}
If $D_{\lambda}$ together with the $P$--arcs is not connected, choose a simple closed curve $\psi$ around each component, separating it from the other components. 
We can choose these curves to be disjoint and to have minimal intersection with the images of the flype circles.
As $D$ is connected, each curve $\psi$ must run through the image of at least one flype circle of $R_{-1}$ as shown in Figure \ref{disjointnesspic7}b. Repeat this process for each component of $\lambda_{-1}$. Let $\Psi$ be the union of the flype circles in $D$ together with those sections of each curve $\psi$ contained in $\lambda_{-1}$.

Colour the $P$--arcs of $R_{-1}$ red, and those of $R_1$ blue. Also colour the components of $\sphere\setminus\Psi$. Colour $\Lambda_{-1}$ red and $\Lambda_1$ blue. Colour a component in $\lambda_{-1}\cap\lambda_1$ blue if it meets $\Lambda_1$ anywhere along its boundary (that is, if it meets a component already coloured blue), and colour it red otherwise.
Thus a point $x$ on $D$ in $\lambda_{-1}\cap\lambda_1$ is in a blue component of $\sphere\setminus\Psi$ if there is a path $\rho$ in $D_{\lambda}$ from $x$ to a flype circle of $R_1$, where $\lambda$ is the component of $\lambda_{-1}$ containing $x$, and is in a red component if no such path exists. 

\smallskip

We wish to apply Proposition \ref{productregionsprop} 
to $R_{-1}$ and $R_1$ to show that they cannot be made disjoint. To do so, we must specify precisely how to position them in $\Sphere\setminus\nhd(L)$. 
We will build up a description of the position of each disc in the constructions of $R_{-1}$ and $R_1$ using the flype circles and the $P$--arcs, working inwards from the edge of the disc. Figure \ref{disjointnesspic11} gives an example of the type of picture we will build up.

First note that $\partial R_{-1}=\partial R_1$.
Let $\overline{\rho}$ denote the closure of $(R_{-1}\cap R_1)\setminus\partial R_1$.
Since $R_{-1},R_1$ are incompressible and $\Sphere\setminus\nhd(L)$ is irreducible, $\overline{\rho}$ can be made to consist only of properly embedded arcs. We will arrange that each such arc is contained in a single disc from the construction of each of $R_{-1},R_1$, with its endpoints on $L$. Let $\partial\overline{\rho}$ denote the (as yet undefined) set of endpoints of these arcs. 

We now arrange the $P$--arcs of each surface. Let $c$ be a crossing in $D$. Then $c$ lies in the interior of $\lambda_{-1}$ or $\lambda_1$, and so has at least one $P$--arc across it on the positive side.

\begin{definition}
Given two special forms, and a crossing $c'$ such that the $P$--arcs around $c'$ are disjoint, call the $P$--arc closest to $c'$ on the positive side \textit{inside} at $c'$, and the other $P$--arc(s) on the positive side of $c'$ \textit{outside}.
\end{definition}

If $c$ only has one $P$--arc across it on the positive side, this $P$--arc must be positioned inside to avoid any intersections of $P$--arcs. 
There are three ways that $c$ could come to have two $P$--arcs across it. One is that $c$ lies in the interior of $\lambda_{-1}\cap\lambda_1$. If the corresponding component of $\sphere\setminus\Psi$ has been coloured blue, put the blue $P$--arc outside, and if it is coloured red then put the red $P$--arc outside.
The second possibility is that one, but not both, of the $P$--arcs across $c$ is a flype arc of a flype circle $\phi$. Without loss of generality, $c$ is in the interior of $\lambda_{-1}$ and $\phi$ is a flype circle of $R_{1}$. Then $c$ lies on the positive side of $\phi$, which contradicts the choice of the position of $\phi$.
Hence this case does not occur.
The third possibility is that the white region of $D$ on the positive side of $c$ is a bigon, and one $P$--arc has been placed on the positive side of $c$, while the other $P$--arc has been placed on the negative side of the other crossing. In this case, put the $P$--arc that is paired with $c$ on the inside.

\smallskip

Consider a flype circle $\phi$ with flype crossing $c$, and the pattern of simple closed curves in $D$ given by $\partial(R_{-1}\cap \bbb{a})$. Let $C_{-1}$ be the curve that runs along the overcrossing arc at $c$. After pushing the endpoints of $P$--arcs that lie on $\phi$
to the positive side of $\phi$, as shown in Figure \ref{disjointnesspic8}, we see that $C_{-1}$ also runs along $\rho_0$. This is because each edge of $\theta(D)$ is only crossed by one overcrossing arc.
\begin{figure}[htbp]
\centering
\input{pictexfiles/disjointnesspic8}
\caption{\label{disjointnesspic8}}
\end{figure}
The same is true of the analogous curve $C_1$ in $\partial(R_{1}\cap \bbb{a})$.

Suppose $\phi$ is a flype circle of $R_{-1}$, so that the positive side of $\phi$ is coloured red. Then the two points where $C_{-1}$ crosses $\phi$ lie on the boundary of the same component of $\sphere\setminus\Psi$ in $\lambda_{-1}$. Let $A$ be the set of all such points where this component of $\sphere\setminus\Psi$ is coloured blue, and define an involution $\widehat{\cdot}\colon A\to A$ such that if $a_1\in A$ then $\widehat{a_1}$ is the other point of $A$ on the same flype circle as $a_1$.
Note that if points $a_1,a_2\in A$ lie on the same simple closed curve $C_{-1}$ of $\partial(R_{-1}\cap \bbb{a})$ then the pairs $a_1,\widehat{a_1}$ and $a_2,\widehat{a_2}$ do not interleave on $C_{-1}$.

We would like it to be the case that no two points of $A$ lie on the same overcrossing arc of $D$. However, this is not in general true (see, for example, Figure \ref{disjointnesspic9}). 
\begin{figure}[htbp]
\centering
\input{pictexfiles/disjointnesspic9}
\caption{\label{disjointnesspic9}}
\end{figure}
We therefore create a subset $a$ of $A$ with a new involution where this does not occur, as follows. Given points $a_1$ and $a_2$ of $A$ that lie on a single overcrossing arc of $D$, remove $a_1$ and $a_2$ and change the involution to pair $\widehat{a_1}$ with $\widehat{a_2}$. 
Repeat this as many times as possible.
Note that it is possible, as in Figure \ref{disjointnesspic9}, to reach a pair where $\widehat{a_1}=a_2$. In such a case the two points are simply deleted from $A$.
Since $A$ is finite, this process will terminate. Notice that the final result will not depend on the order in which pairs of points are chosen. Let $a$ be the resulting set of points, and $\overline{\cdot}\colon a\to a$ the resulting involution. 
The following three properties of $A$ are also true in $a$. 
If $a_1\in a$ then $a_1$ and $\overline{a_1}$ lie on the same curve $C_{-1}$ of $\partial(R_{-1}\cap \bbb{a})$ and on the same curve $C_1$ of $\partial(R_{1}\cap \bbb{a})$. Additionally, $a_1$ and $\overline{a_1}$ lie on the boundary of the same component of $\sphere\setminus\Psi$ in $\lambda_{-1}\cap\lambda_1$. Furthermore, if $a_2\in a$ then the pairs $a_1,\overline{a_1}$ and $a_2,\overline{a_2}$ do not interleave on $C_{-1}$ or on $C_1$.

The same process in $\bbb{b}$ gives another set $b$ with involution $\overline{\cdot}\colon b \to b$. We will end up with $\partial\overline{\rho}=a\cup b$.

Armed with these, we next connect the ends of the $P$--arcs to give the position of the neighbourhood of the boundary of every disc in the construction of $R_{-1}$ and $R_1$. For each disc, this neighbourhood is an annulus, one boundary component of which we have already positioned (as $\partial R_{-1}$ and $\partial R_1$ are the same fixed curve on $\partial\!\nhd(L)$). We describe the relative positions of the annuli by drawing on $D$ the other boundary curve of each annulus. We will describe this process in $\bbb{a}$. That in $\bbb{b}$ is analogous. 

Consider two crossings $c$ and $c'$ of $D$ that are adjacent in $D$. Suppose the same colour $P$--arc lies inside at each, as shown in Figure \ref{disjointnesspic10}a. We can then connect the $P$--arcs along the overcrossing arc at $c'$ without them crossing (see Figure \ref{disjointnesspic10}a).
\begin{figure}[htb]
\centering
(a)
\input{pictexfiles/disjointnesspic10a}
(b)
\input{pictexfiles/disjointnesspic10b}
\caption{\label{disjointnesspic10}}
\end{figure}
Suppose instead that the $P$--arcs inside at $c$ and $c'$ are not the same colour. Then one of the crossings lies in $\Lambda_{-1}$ whereas the other is either in $\Lambda_1$ or in a blue region of $\lambda_{-1}\cup\lambda_1$. Hence exactly one point $a_1$ of $a$ lies between $c$ and $c'$. This time we connect the $P$--arcs so that the curves created cross once at $a_1$, as shown in Figure \ref{disjointnesspic10}b.
In this way we can connect all the $P$--arcs to form two sets of simple closed curves, as required. Figure \ref{disjointnesspic11} is an example of the result of this process.

\begin{figure}[htb]
\centering
\input{pictexfiles/disjointnesspic11}
\caption{\label{disjointnesspic11}}
\end{figure}

Finally we position the interior of each disc of $R_{-1}$ and $R_1$. We may view those of $R_{-1}$ as fixed, and vary those of $R_1$. In $\bbb{a}$ we wish to arrange the discs such that an arc of $\overline{\rho}$ connects points $a_1,a_2\in a$ if and only if $\overline{a_1}=a_2$.
Recall that $a_1$ and $\overline{a_1}$ lie on the same curves of $\partial(R_{-1}\cap\bbb{a})$ and $\partial(R_{-1}\cap\bbb{a})$ for any $a_1\in a$, and for $a_3\in a$ the pairs $a_1,\overline{a_1}$ and $a_3,\overline{a_3}$ do not interleave along these curves.
The same holds in $\bbb{b}$. Thus there is no obstruction to our choice of $\overline{\rho}$. This completes the construction of $R_{-1}\cup R_1$.

\bigskip

It now remains to check that $R_{-1}$ and $R_1$ have simplified intersection. Then, since $R_{-1}$ and $R_1$ are not disjoint, Proposition \ref{productregionsprop} shows that $\dist_{\ms(L)}(R_{-1},R_1)>1$. Let $M$ be the complement in $\Sphere\setminus\nhd(L)$ of a regular neighbourhood of $R_{-1}\cup R_1$.  

Suppose a component $M'$ of $M$ does not meet $\sphere$. Then $\partial M'$ is composed of sections of the discs in the construction of $R_{-1}$ and $R_1$, with at least one from each surface. Consider a disc $S$ of $R_1$ that meets $\partial M'$. Then $\partial(S\cap\partial M')$ is a collection of simple closed curves in $R_{-1}\cap R_1$.  Without loss of generality, assume $M'\subset \bbb{a}$. Then $\partial(S\cap\partial M')$ is made up of overcrossing arcs of $L$ and arcs in $\overline{\rho}$. As $L$ is not the unknot, any simple closed curve of $\partial(S\cap\partial M')$ includes both types of arc. This is not possible, since no overcrossing arc includes more than one point of $\partial\overline{\rho}$.
Hence every component of $M$ meets $\sphere$. This means that, in checking whether any component is a product region, it is enough to consider those that meet each region of $D$.

\smallskip

Consider a point of $\partial\overline{\rho}$ on $\partial\!\nhd(L)$. Near this point, $M$ has three components, only one of which, $M_L$, meets $\partial\!\nhd(L)$. On $\partial M_L$ we see the pattern shown in Figure \ref{disjointnesspic12}. 
\begin{figure}[htbp]
\centering
\input{pictexfiles/disjointnesspic12}
\caption{\label{disjointnesspic12}}
\end{figure}
Suppose $M_L$ is a product region between $R_{-1}$ and $R_1$. 
Then a closed regular neighbourhood of $\overline{\rho}\cup\partial\!\nhd(L)$ in $\partial M_L$ 
is of the form $\partial S_L\times\intvl$ for some surface $S_L$, and in particular is a collection of annuli. Since any component of $\partial\!\nhd(L)\cap\partial M_L$ is an annulus, and each endpoint of an arc of $\overline{\rho}$ lies on $\partial\!\nhd(L)$, this cannot be the case. Hence $M_L$ is not a product region.

Let $r$ be a black region of $D$, and let $M_r$ be the component of $M$ that meets $r$. Then $M_r$ meets the Seifert surface $R_0$. Since $R_0$ is disjoint from $R_{-1}\cup R_1$ and connected, $R_0$ is entirely contained in $M_r$. The manifold $M_L$ also meets $R_0$. This shows that $M_r=M_L$, so $M_r$ is not a product region. The same is true if $r$ is instead a section of a white region of $D$ that meets $L$. We are therefore left to consider those sections of white regions of $D$ that are entirely bounded by $P$--arcs. These components of $M$ form sutured manifolds, where the sutures are $R_{-1}\cap R_1$.

\smallskip

Let $r$ be a section of a white region of $D$ that lies between $P$--arcs in $\Lambda_{-1}\cup\Lambda_1$, and let $M_r$ be the component of $M$ that meets it. Let $\Lambda$ be the component of $\sphere\setminus\Psi$ containing $r$. Inside $\Lambda$, the discs of $R_{-1}$ and $R_1$ are parallel to $R_0$, and $M_r$ is isotopic to the complement of $R_0$ there. As in the proof of Proposition \ref{disjointnessprop1}, we aim to decompose $M_r$ along product discs, and to edit the diagram $D$ to find a special alternating link diagram $D'$ such that one piece of $M_r$ is the complement of the surface given by applying Seifert's algorithm to $D'$. Again we will see that the link with diagram $D'$ is not fibred, and hence that $M_r$ is not a product sutured manifold. 

Let $c$ be the flype crossing of a flype circle $\phi$ in the boundary of $\Lambda$. 
Let $C_{-1}$ be the curve of $\partial(R_{-1}\cap \bbb{a})$ that crosses $c$ and $C_1$ the corresponding curve of $\partial(R_{1}\cap \bbb{a})$. Then $C_{-1}$ and $C_1$ also cross the flype arc of $\phi$ together. Let $S_i$ be the disc of $R_i\cap \bbb{a}$ with boundary $C_i$ for $i=\pm 1$. Choose a point $x$ on $C_{-1}\cap C_1$ slightly inside $\Lambda$ from $c$, and a similar point $x'$ near the flype arc of $\phi$. 

Suppose that there exists a point $a_1$ of $a$ on $C_{-1}$ or on $C_1$ such that $a_1,\overline{a_1}$ interleave with $x,x'$. 
Then there is a point $a_2\in A$ such that $a_2,\widehat{a_2}$ interleave with $x,x'$. This contradicts that all the flype circles in $\Phi$ are disjoint. Thus no such point exists.

This means there is a product disc in $M_r$ between $S_{-1}$ and $S_1$ meeting $R_{-1}\cap R_1$ at $x$ and $x'$. Similarly there is a corresponding product disc in $M_r\cap \bbb{b}$. Cut $M_r$ along these product discs, and retain the component that meets $r$. The corresponding change in $D$ is to collapse the negative side of $\phi$ to a single crossing, as in Figure \ref{disjointnesspic3}. After so collapsing all flype circles that bound $\Lambda$, we arrive at a diagram $D'$ as described above. By the same reasoning as in the proof of Proposition \ref{disjointnessprop1}, the remaining section of $M_r$ is not a product region, and so $M_r$ is not.

\smallskip

Finally, let $r$ be a section of a white region of $D$ that lies between a red $P$--arc and a blue $P$--arc parallel to it, and let $M_r$ be the component of $M$ that meets it. In this case, we show by contradiction that $M_r$ also meets $\sphere$ in a white region between $P$--arcs in $\Lambda_{-1}\cup\Lambda_1$, and hence it has already been shown to not be a product region. 

Suppose $M_r$ does not meet $\sphere$ in any such white region.  
Let $\lambda$ be the component of $\sphere\setminus\Psi$ containing $r$. Note that every crossing in the interior of $\lambda$ has both a red $P$--arc and a blue $P$--arc across it on the positive side.
Since $D$ is connected, there is a path $\sigma$ from $r$ to a flype circle $\phi$ on the boundary of $\lambda$ that is contained in $D$ and the $P$--arcs. If $\sigma$ passes between an overcrossing and an undercrossing at a crossing $c$ in the interior of $\lambda$, change it to instead switch between the two via the $P$--arcs across the positive side of $c$. After such changes, $\sigma$ gives a path in $\partial M_r$. 

Given that $M_r$ does not meet the white region between the $P$--arcs on the other side of the flype crossing or flype arc, the endpoint of $\sigma$ must be a point of $\partial\overline{\rho}=a\cup b$. 
From this we see that $\lambda$ is coloured blue, 
which means there is a path $\sigma$ in $D_{\lambda'}$ with the $P$--arcs from $r$ to a flype circle of $R_1$, where $\lambda'$ is the component of $\lambda_{-1}$ containing $\lambda$. Again we can make this path run along $P$--arcs rather than moving directly between overcrossings and undercrossings.
Our aim is to show that $\sigma$ gives a path in $\partial M_r$. Then $M_r$ will meet a white region of $D$ between the $P$--arcs in $\Lambda_1$, contradicting our assumption.

We know that $\sigma$ gives a path in $\partial M_r$ until it first reaches a flype circle of $R_{-1}$, as $D$ and $D_{\lambda'}$ are the same in the interior of $\lambda_{-1}$. Without loss of generality, suppose $\sigma$ meets a flype circle $\phi$ of $R_{-1}$ on an overcrossing arc of $D$. 
There it meets a point $a_1\in a$. 
Then $\partial M_r$ contains the arc of $\overline{\rho}$ that runs from $a_1$ to $\overline{a_1}$.
Note that the path $\sigma$ must also run through the point of $D_{\lambda'}$ that corresponds to $\overline{a_1}$. 

In this way, we can see that $\sigma$ defines a path in $\partial M_r$ as required. Therefore $M_r$ is not a product region.
\end{proof}

\section{Non-special links}\label{nonspecialsection}

Theorem 1.2 of \cite{MR1664976} gives two families of links that illustrate the importance of assuming, in Theorem \ref{isomorphismtheorem}, that the link $L$ has a special alternating diagram, and not just an alternating diagram. These families of links are the subjects of the next two propositions.

\begin{proposition}
If $L_{\gamma_n}$ is the link shown in Figure \ref{nonspecialpic1} then $\ms(L_{\gamma_n})$ contains an $(n-1)$--simplex in which exactly one vertex is given by applying Seifert's algorithm to an alternating diagram.
\begin{figure}[htbp]
\centering
\input{pictexfiles/nonspecialpic1}
\caption{\label{nonspecialpic1}}
\end{figure}
\end{proposition}
\begin{proof}
There is only one non-trivial flype possible on the diagram given, and it is not an essential flype. 
Hence only one minimal genus Seifert surface, $R$, for $L_{\gamma_n}$ comes from an alternating diagram.

There is a product disc decomposition of $R$ that removes the plumbed on Hopf band. The resulting surface, $R'$, is given by applying Seifert's algorithm to the special alternating link $L'_{\gamma_n}$ shown in Figure \ref{nonspecialpic2}. By Theorem \ref{isomorphismtheorem} the link of the vertex $R'$ in $\ms(L'_{\gamma_n})$ is an $(n-2)$--simplex, and so the same is true of $R$ in $\ms(L_{\gamma_n})$ by Proposition \ref{samelinkprop}.
\begin{figure}[htbp]
\centering
\input{pictexfiles/nonspecialpic2}
\caption{\label{nonspecialpic2}}
\end{figure}
\end{proof}

\begin{proposition}
If $L_{\delta_n}$ is the link shown in Figure \ref{nonspecialpic3}
then $\ms(L_{\delta_n})$ is as shown in Figure \ref{nonspecialpic4}, where a white circle represents a surface given by applying Seifert's surface to an alternating diagram, and a black circle represents a surface that cannot be so constructed.
\begin{figure}[htbp]
\centering
\input{pictexfiles/nonspecialpic3}
\caption{\label{nonspecialpic3}}
\end{figure}
\begin{figure}[htbp]
\centering
\input{pictexfiles/nonspecialpic4}
\caption{\label{nonspecialpic4}}
\end{figure}
\end{proposition}
\begin{proof}
Again there are no essential flypes of this diagram, so only one minimal genus Seifert surface comes from an alternating diagram.

The surface $R_m$, for $0\leq m\leq n$, is the surface given by applying Seifert's algorithm to the special (but not alternating) diagram $D_m$ shown in Figure \ref{nonspecialpic5}. In Figure \ref{nonspecialpic3} we may see  each of the horizontal twisted bands (which are plumbed on Hopf bands) as lying either above or below the rest of the surface. This does not change the ambient isotopy class of the Seifert surface because the Hopf link is fibred.  From this we can see that $R_0$ is the surface shown in Figure \ref{nonspecialpic3} (see Figure \ref{nonspecialpic6}).
\begin{figure}[htbp]
\centering
\input{pictexfiles/nonspecialpic5}
\caption{\label{nonspecialpic5}}
\end{figure}
\begin{figure}[htbp]
\centering
\input{pictexfiles/nonspecialpic6}
\caption{\label{nonspecialpic6}}
\end{figure}

Consider the diagram $D_0$. We can express $R_1$ on this diagram as a surface in special form on an alternating diagram, with another surface attached in the obvious way as shown in Figure \ref{nonspecialpic7}.
\begin{figure}[htbp]
\centering
\input{pictexfiles/nonspecialpic7}
\caption{\label{nonspecialpic7}}
\end{figure}
Thus we see that $R_1$ is disjoint from $R_0$. We want to show that $R_0\neq R_1$, and that no other minimal genus Seifert surface is disjoint from $R_0$. 
To do this, we show that the link of $R_0$ in $\ms(L_{\delta_n})$ consists of only one vertex, and this vertex is $R_1$. 

The white region marked $*$ in Figure \ref{nonspecialpic7} defines a product disc in the complementary sutured manifold to $R_0$. 
Doing the product disc decomposition gives the same diagram, but with one fewer twisted bands in the right-hand surface. After $n$ such decompositions, we arrive at a surface $R^{(n)}_0$ given by just the alternating section of the diagram. 
By Theorem \ref{isomorphismtheorem} the link of $R^{(n)}_0$ is a single vertex. Hence, by repeated application of Proposition \ref{samelinkprop}, the link of $R_0$ is also a single vertex, as required. By following the constructions in Theorem \ref{isomorphismtheorem} and Proposition \ref{samelinkprop} we find that this single vertex is $R_1$.

We now show, for $1\leq m \leq n-1$, that the link of $R_m$ is $\{R_{m-1},R_{m+1}\}$.
Again we find that these surfaces can be expressed as a surface in special form on an alternating diagram together with a second surface attached in the obvious way, as shown in Figure \ref{nonspecialpic8}a and b respectively.
\begin{figure}[htbp]
\centering
\begin{tabular}{@{}l@{}l@{}}
(a)&
\input{pictexfiles/nonspecialpic8a}\\
(b)&
\input{pictexfiles/nonspecialpic8b}
\end{tabular}
\caption{\label{nonspecialpic8}}
\end{figure}
Also as before, after $n$ product disc decompositions we may apply Theorem \ref{isomorphismtheorem}, giving that the link of $R^{(n)}_{m}$ contains exactly two vertices and no edges.
This shows that the link of $R_m$ is as required.

It remains to consider the surface $R_n$. We now know that $R_{n-1}$ is in the link of $R_n$. Thus we need only show that the link of $R_n$ contains only one element.
Each white bigon region of $D_n$ defines a product disc in the complement of $R_n$. Using $n-1$ such product disc decompositions, removing all but one of the once-twisted bands, and an isotopy, we arrive at the diagram and surface shown in Figure \ref{nonspecialpic9},
\begin{figure}[htbp]
\centering
\input{pictexfiles/nonspecialpic9}
\caption{\label{nonspecialpic9}}
\end{figure}
where the horizontal twisted band lies above the rest of the surface. Thus we are interested in a plumbing of two surfaces, $S_a$ and $S_b$, shown in Figure \ref{nonspecialpic10}. Call the plumbed surface $S_{ab}$.

\begin{figure}[htbp]
\centering
\input{pictexfiles/nonspecialpic10}
\caption{\label{nonspecialpic10}}
\end{figure}

$S_a$ is an annulus with two full twists, bounding a torus link $L_a$.
After $n$ further disc decompositions, removing the twisted bands from the right-hand side of the diagram, $S_b$ is also reduced to an annulus with two full twists. From this we see that both $S_a$ and $S_b$ are unique minimal genus Seifert surfaces for $L_a$ and $L_b=\partial S_b$ respectively, and also that neither link is fibred. 
This is enough to show there is at least one minimal genus Seifert surface for $\partial S_{ab}$ in the link of $S_{ab}$. 
However, we already knew that such a surface exists (the image of $S_{n-1}$ under the product disc decompositions), and we have not yet shown there are no others.

To do this, we consider the arcs $\rho_a$ and $\rho_b$ shown in Figure \ref{nonspecialpic10}. These should be seen as lying in the boundary of the complementary sutured manifolds to the surfaces $S_a,S_b$. Note that $\rho_a,\rho_b$ both lie on the upper side of the surfaces as shown.
By \cite{MR2131376} Proposition 3.4, the proof will be complete if we show that there is no product disc with either $\rho_a$ or $\rho_b$ as part of its boundary. Since $\rho_a$ and $\rho_b$ are essential, any such product disc will be essential. This result is clear for $\rho_a$ as there are no essential product discs in the complement of $S_a$.

As we reduce $S_b$ to an annulus by product disc decompositions, what is the effect on $\rho_b$? Figure \ref{nonspecialpic11} shows the result of the first decomposition.
\begin{figure}[htbp]
\centering
\input{pictexfiles/nonspecialpic11}
\caption{\label{nonspecialpic11}}
\end{figure}
This leaves two copies of the arc $\rho'_b$, the curve in $S'_b$ analogous to $\rho_b$ in $S_b$. Inductively, we are therefore only concerned with the curve $\rho^{(m)}_b$ in $S^{(m)}_b$. The final pair $(S^{(n)}_b,\rho^{(n)}_b)$ is the same as $(S_a,\rho_a)$, so there is no essential product disc with $\rho^{(n)}_b$ contained in its boundary. By $n$ applications of Lemma \ref{productdiscslemma}, there is no product disc $T$ with $\rho_b\subset \partial T$.
\end{proof}

\begin{remark}
These two propositions together complete the proof of \cite{MR1664976} Theorem 1.2.
\end{remark}

\section{The realisation of $\ms(L)$}\label{productssection}

Let $L$ be a prime link with a reduced, special alternating diagram $D$. In this section we turn our attention to Theorem 1.6 of \cite{MR1664976}, which asserts that $\ms(L)$ is homeomorphic to a ball in $\mathbb{R}^n$ for some $n$. The results we will use to prove this refer to simplicial complexes with a partial ordering on the vertices. Thus we will need to give such an order to the vertices of $\kc{D}$. 

Note that the simplices of $\kc{D}$ already have a cyclic ordering on their vertices. We need to break this circle to give a linear order. To do this, pick a region $r$ of $\theta(D)$. We break each circle in the ordering at the region $r$. That is, if $v_1$ is obtained from $v_0$ by adding the region $r$ then we remove the relation $v_0<v_1$. Note that the order this gives to the vertices of a face of a simplex matches that induced by the ordering of the vertices of the whole simplex. This method also orders the edges of each $\theta$--graph in $\theta(D)$.

\begin{definition}[\cite{MR1802191}]
Fix $m,n\in\mathbb{N}$, and let $v_0,\ldots,v_n$ be the (ordered) vertices of an $n$--simplex $[v_0,\ldots,v_n]$.

A \textit{colour scheme} is an $m \times (l+1)$ matrix $\mathbf{X}=(x_{ij})$ for some $l\leq n$ with $x_{ij}\in\{0,\ldots,n\}$. In addition, $\mathbf{X}$ must have pairwise distinct columns, and its entries must satisfy $x_{10}\leq x_{11}\leq \cdots\leq x_{1l}\leq x_{20}\leq\cdots\leq x_{ml}$.
For $k\leq l$, column $k$ of $\mathbf{X}$ defines a point $v'_k$ of $[v_0,\ldots,v_n]$ by $v'_k=(v_{x_{1k}}+\cdots+v_{x_{mk}})/m$. Thus $\mathbf{X}$ defines an $l$--simplex $[v'_0,\ldots,v'_l]$.

The \textit{edgewise subdivision} $\esd_m([v_0,\ldots,v_n])$ of $[v_0,\ldots,v_n]$ is made up of all simplices given by colour schemes.
\end{definition}

\begin{remark}
In \cite{MR1802191} Section 3 it is proved that $\esd_m([v_0,\ldots,v_n])$ is indeed a subdivision of $[v_0,\ldots,v_n]$.
\end{remark}

\begin{proposition}[\cite{MR1664976} Theorem 1.6(1)]\label{onethetagraphprop}
Suppose $\theta(D)$ consists of a single $\theta$--graph with $n+1$ edges and total weight $m$. Then $\kc{D}$ is the edgewise subdivision of an $n$--simplex.
\end{proposition}
\begin{proof}
Fix a region $r_0$ of $\theta(D)$ to give an ordering on the vertices of $\kc{D}$ as above (note that in this case $\tilde{\theta}(D)$ is $\theta(D)$).
Label the edges of $\theta(D)$ as $e_0,\ldots,e_n$ where $e_0<e_1<\cdots<e_n$. Label the other regions of $\theta(D)$ in order around the $\theta$--graph in the positive direction, so that $\partial^-r_i=e_{i-1}$ and $\partial^+r_i=e_i$.

Let $v_0,\ldots,v_n$ be the (ordered) vertices of an $n$--simplex $[v_0,\ldots,v_n]$.
Given a vertex $u=(w_0,\ldots,w_n)$ of $\kc{D}$, we may define an $m\times 1$ matrix $\mathbf{X}=(x_{ij})$ by taking $x_{i0}=k$ where $\sum_{j< k}w_j<i\leq \sum_{j\leq k}w_j$.  
That is, the number of times $k$ occurs in $\mathbf{X}$ is the weight that $u$ assigns to $e_k$, and the entries of $\mathbf{X}$ are arranged so as to be non-decreasing.
The matrix $\mathbf{X}$ is then a colour scheme, and so defines a vertex of $\esd_m([v_0,\ldots,v_n])$. Thus there is a map $B$ from the vertices of $\kc{D}$ to the vertices of $\esd_m([v_0,\ldots,v_n])$. It is clear that $B$ is a bijection. We will show that $B$ induces a bijection on the $n$--simplices of the two complexes. This will complete the proof as in each of the complexes every simplex is a face of an $n$--simplex.

Let $u_0,\ldots,u_n$ be the vertices of an $n$--simplex in $\kc{D}$ with $u_0<u_1<\cdots<u_n$. Then there is a permutation $\pi\colon\{1,\ldots,n\}\to\{1,\ldots,n\}$ such that $u_i$ is obtained from $u_{i-1}$ by adding the region $r_{\pi(i)}$. 
Let $\mathbf{X}$ be an $m\times (n+1)$ matrix, where column $k$ of $\mathbf{X}$ is the column vector defined by $u_k$ in the construction of the map $B$. The column vectors from $u_k$ and $u_{k+1}$ differ exactly in one coordinate, which changes from $\pi(k)-1$ to $\pi(k)$. 
Because $u_0,\ldots,u_n$ is an $n$--simplex, we also know that $u_0$ is obtained from $u_n$ by adding $r_0$. This has the effect of dropping the last coordinate of the column vector, moving the remaining entries down one place, and inserting a 0 in the top coordinate.
Hence $\mathbf{X}$ is a colour scheme defining an $n$--simplex of $\esd_m([v_0,\ldots,v_n])$.

Conversely, choose an $n$--simplex of $\esd_m([v_0,\ldots,v_n])$ and let $\mathbf{X}$ be the colour scheme defining it. Let $u_k$ be the vertex of $\kc{D}$ given by column $k$ of $\mathbf{X}$.  That is, $u_k$ assigns a weight of 1 to $e_i$ for each time $i$ occurs in the column vector. As the columns of $\mathbf{X}$ are pairwise distinct, at least one element changes between $u_k$ and $u_{k+1}$. The ordering of the values of the elements of $\mathbf{X}$ ensure this is an increase in each case, and that the sum of the sizes of these increases is at most $n$. Thus exactly one coordinate increases between $u_k$ and $u_{k+1}$ and this coordinate increases by 1. Moreover, for each $k\leq n$, exactly one of these changes is from $k-1$ to $k$, and $x_{(k-1)n}=x_{k0}$ for $1<k\leq m$ while $x_{10}=0$ and $x_{mn}=n$. Define a permutation $\pi\colon\{1,\ldots,n\}\to\{1,\ldots,n\}$ by $\pi(j)=k$ if the move from $u_{j-1}$ to $u_j$ sees a coordinate change from $k-1$ to $k$. Then $u_{j}$ is obtained from $u_{j-1}$ by adding the region $r_{\pi(j)}$ for $j\leq n$ and adding $r_0$ to $u_n$ gives $u_0$. This shows that $u_0,\ldots,u_n$ span an $n$--simplex in $\kc{D}$.

Each of these two maps is an injection, as simplices are defined by their vertices. As there are only finitely many $n$--simplices in either complex, both maps are therefore bijections.
\end{proof}

\begin{definition}[\cite{MR0050886} Chapter II Definition 8.7]
A simplicial complex $\mathcal{X}$ is \textit{ordered} if there is a binary relation $\leq$ on the vertices of $\mathcal{X}$ with the following properties.
\begin{itemize}
 \item[(P1)] $(u \leq v \textrm{ and } v\leq u)\Rightarrow u=v$.
 \item[(P2)] If $u,v$ are distinct, $(u \leq v \textrm{ or } v\leq u) \Leftrightarrow \textrm{ $u$ and $v$ are adjacent}$.
 \item[(P3)] If $u,v,w$ are vertices of a 2--simplex then $(u \leq v \textrm{ and } v\leq w)\Rightarrow u\leq w$.
\end{itemize}
\end{definition}

\begin{remark}
It is clear that in searching for such a relation we may use the following weaker version of (P2).
\begin{itemize}
 \item[(P2)$'$] If $u,v$ are adjacent then $(u\leq v \textrm{ or } v\leq u)$. 
\end{itemize}
To see this, note that we can remove all relationships between non-adjacent vertices.
\end{remark}

\begin{definition}[\cite{MR0050886} Chapter II Definition 8.8]
Let $\mathcal{X}_1,\mathcal{X}_2$ be ordered simplicial complexes. We define the simplicial complex $\mathcal{X}_1\times\mathcal{X}_2$. Its vertices are given by the set $\V(\mathcal{X}_1)\times\V(\mathcal{X}_2)$. Vertices $(u_0,v_0),\ldots,(u_n,v_n)$ span an $n$--simplex if the following hold.
\begin{itemize}
 \item $\{u_0,\ldots,u_n\}$ is an $m$--simplex of $\mathcal{X}_1$ for some $m\leq n$.
 \item $\{v_0,\ldots,v_n\}$ is an $m$--simplex of $\mathcal{X}_2$ for some $m\leq n$.
 \item The relation defined by $(u,v)\leq (u',v')\Leftrightarrow(u\leq u' \textrm{ and } v\leq v')$ gives a total linear order on $(u_0,v_0),\ldots,(u_n,v_n)$.
\end{itemize}
\end{definition}

\begin{remark}
The projection maps on the vertices extend to simplicial maps of the complexes.
\end{remark}

\begin{theorem}[\cite{MR0050886} Chapter II Lemma 8.9]\label{productcomplexthm}
The map $|\mathcal{X}_1\times\mathcal{X}_2|\to|\mathcal{X}_1|\times|\mathcal{X}_2|$ induced by projection is a homeomorphism.
\end{theorem}

\begin{remark}
With the ordering on the vertices defined above, $\kc{D}$ is an ordered simplicial complex.
\end{remark}

\begin{theorem}[\cite{MR1664976} Theorem 1.6(2)]
The realisation $|\kc{D}|$ of $\kc{D}$ is homeomorphic to a ball of dimension $\sum{(n^{\theta}(e^{\theta})-1)}$, where the sum is taken over all $\theta$--graphs $e^{\theta}$ in $\theta(D)$ and $n^{\theta}(e^{\theta})$ is the number of edges in $e^{\theta}$.
\end{theorem}
\begin{proof}
We proceed by induction on the number of $\theta$--graphs in $\theta(D)$. 
If there are no $\theta$--graphs then $\kc{D}$ is a single vertex.
The case of one $\theta$--graph is covered by Proposition \ref{onethetagraphprop}.

The construction of the 
simplicial complex $\kc{D}$ is dependent only on $\tilde{\theta}(D)$ together with a choice of positive direction and total weight on each $\theta$--graph. Let $\Theta$ be the set of such graphs. That is, each element of $\Theta$ is a finite collection of disjoint $\theta$--graphs in $\sphere$, with a choice of positive direction and total weight on each. Then for every such graph $\theta$ we can construct a simplicial complex $\kc{\theta}$. It is on this set of complexes that we will induct. Note that the base cases hold in this more general setting.

Now suppose that the result holds for all elements of $\Theta$ with at most $(m-1)$ $\theta$--graphs. Let $\theta_0$ be one with $m$ $\theta$--graphs. Choose a region $r$ of $\sphere\setminus\theta_0$ that meets at least two $\theta$--graphs, and pick a simple closed curve $\rho$ contained in $r$ that separates the $\theta$--graphs of $\theta_0$. This gives two new elements of $\Theta$, each with at most $(m-1)$ $\theta$--graphs. Call these $\theta_1,\theta_2$. Use the region $r$ to order the vertices of $\kc{\theta},\kc{\theta_1},\kc{\theta_2}$ as above. Then, by Theorem \ref{productcomplexthm}, 
$|\kc{\theta_1}\times\kc{\theta_2}|\cong|\kc{\theta_1}|\times|\kc{\theta_2}|$. The inductive hypothesis gives that $|\kc{\theta_1}|,|\kc{\theta_2}|$ are balls of the relevant dimension, and so the result holds for $|\kc{\theta_1}\times\kc{\theta_2}|$.

It remains to show that $\kc{\theta_1}\times\kc{\theta_2}$ is $\kc{\theta_0}$. Clearly the obvious map on the vertices is a bijection. Again we will check that it induces a bijection on the top-dimensional simplices.

Consider an ordered $n$--simplex $[w_0,\ldots,w_n]$ in $\kc{\theta_0}$. Then $w_0$ assigns a weight to each edge of $\theta_0$. There is an ordering of the regions of $\sphere\setminus\theta_0$ with $r$ last such that the vertices of $[w_0,\ldots,w_n]$ are given by adding these regions in turn in order. This induces similar orderings of the regions of $\sphere\setminus\theta_i$.
For $0\leq i\leq n$, write $w_i=(u_i,v_i)$, where $u_i$ gives the weights on the edges of $\theta_1$ and $v_i$ gives the weights on the edges in $\theta_2$. If $i<j$ then $u_i\leq u_j$ and $v_i\leq v_j$. Thus $[(u_0,v_0),\ldots,(u_n,v_n)]$ is a simplex in $\kc{\theta_1}\times\kc{\theta_2}$.

Conversely, consider an $n$--simplex $[(u_0,v_0),\ldots,(u_n,v_n)]$ in $\kc{\theta_1}\times\kc{\theta_2}$, where the $u_i$ are vertices of $\kc{\theta_1}$, the $v_i$ are vertices of $\kc{\theta_2}$ and $(u_0,v_0)<(u_1,v_1)<\cdots<(u_n,v_n)$. 
For $0<i\leq n$, as $u_{i-1}\leq u_i$, either $u_{i-1}=u_i$ or $u_i$ is given by adding some regions of $\sphere\setminus\theta_1$ other than $r$ to $u_{i-1}$. Similarly, either $v_{i-1}=v_i$ or $v_i$ is given by adding some regions of $\sphere\setminus\theta_2$ other than $r$ to $v_{i-1}$.
Since there are only $n$ regions of $\sphere\setminus\theta_0$ other than $r$, for $0< i\leq n$ either $u_i=u_{i-1}$ and $v_i$ is given by adding a single region to $v_{i-1}$ or instead $v_i=v_{i-1}$ and $u_i$ is given by adding a single region to $u_{i-1}$. Because $[u_0,\ldots,u_n]$ is a simplex in $\kc{\theta_1}$, adding $r$ to $u_n$ gives $u_0$. Likewise, adding $r$ to $v_n$ gives $v_0$. 
Thus there is an ordering on the regions of $\sphere\setminus \theta_0$ with $r$ last that shows that $[(u_0,v_0),\ldots,(u_n,v_n)]$ is an $n$--simplex in $\kc{\theta_0}$.
\end{proof}


\appendix
\section{Proof by picture: a worked example of Proposition \ref{disjointnessprop2}}

Let $L_{\varepsilon}$ be the link with minimal genus Seifert surfaces $R_{-1}$ and $R_1$ given in admissible special form in Figure \ref{disjointnessegpic1}a and b respectively.
\begin{figure}[htbp]
\centering
\begin{tabular}{@{}l@{}l@{}}
(a)&
\input{pictexfiles/disjointnessegpic1a}\\
(b)&
\input{pictexfiles/disjointnessegpic1b}\\
(c)&
\input{pictexfiles/disjointnessegpic1c}
\end{tabular}
\caption{\label{disjointnessegpic1}}
\end{figure}
We show the steps of the proof of Proposition \ref{disjointnessprop2} in this case, as follows.
\begin{itemize}
  \item[i.] Construct a minimal set of flype circles. Perform flypes to remove $\Lambda_{-1}\cap\Lambda_1$. Change flype circles to equivalent ones where needed. This gives Figure \ref{disjointnessegpic1}c.
  \item[ii.] For each region $\lambda$ of $\lambda_{-1}$, construct the diagram $D_{\lambda}$ and, if it has more than one component, choose a curve $\psi$ around each component. The two new diagrams are shown in Figure \ref{disjointnessegpic2}a and b.
  \item[iii.] Colour the components of $\sphere\setminus\Psi$. Define $a$ and $b$, and the corresponding involutions. See Figure \ref{disjointnessegpic2}c (the pairs of points in $a$ are marked with stars).
  \item[iv.] Position the $P$--arcs and connect them. Figure \ref{disjointnessegpic3} shows the resulting picture in $\bbb{a}$.
\end{itemize}
\begin{figure}[htbp]
\centering
\begin{tabular}{@{}l@{}c@{}}
(a)&
\input{pictexfiles/disjointnessegpic2a}\\
(b)&
\input{pictexfiles/disjointnessegpic2b}\\
(c)&
\input{pictexfiles/disjointnessegpic2c}
\end{tabular}
\caption{\label{disjointnessegpic2}}
\end{figure}
\begin{figure}[htbp]
\centering
\input{pictexfiles/disjointnessegpic3}
\caption{\label{disjointnessegpic3}} 
\end{figure}
\label{firstappendix}


\bibliography{hsreferences}
\bibliographystyle{hplain}


\bigskip

\noindent
Mathematical Institute

\nopagebreak\noindent
University of Oxford

\nopagebreak\noindent
24--29 St Giles'

\nopagebreak\noindent
Oxford OX1 3LB

\nopagebreak\noindent
England

\nopagebreak\smallskip
\nopagebreak\noindent
\textit{jessica.banks[at]lmh.oxon.org}

\end{document}